%% file: capa.tex
\documentclass[11pt,a4paper,reqno]{amsart}
\usepackage{amsmath,amsthm,amsfonts,amssymb,bm,wasysym}
\usepackage{caption}
\usepackage{subcaption}
\usepackage{comment}
\usepackage{fullpage,bbm}
\usepackage{graphicx}
\usepackage{tikz,ifthen}
\usepackage{hyperref}
\usetikzlibrary{calc}
\usetikzlibrary{patterns}
\usetikzlibrary{arrows}
\usepackage{mathtools}
\usetikzlibrary{decorations.pathreplacing}
\theoremstyle{plain}
\newtheorem{theorem}{Theorem}
\newtheorem{proposition}[theorem]{Proposition}

\newtheorem{lemma}[theorem]{Lemma}

\usepackage{pgfplots}

\theoremstyle{definition}
\newtheorem{definition}[theorem]{Definition}

\theoremstyle{remark}

\begin{document}

\def\MLine#1{\par\hspace*{-\leftmargin}\parbox{\textwidth}{\[#1\]}}

\def\A{\mathbb{A}}
\def\Ab{\mathcal{A}b}
\def\absq{{a^{\prime}}^2+{b^\prime}^2}
\def\Adir{\mathcal{A}_{\mathsf{dir}}}
\def\AP{\text{G}}
\def\app{{a^{\prime\prime}}^2+1}
\def\argmin{\text{argmin}}
\def\arb{arbitrary }
\def\ass{assumption}
\def\arrow{\rightarrow}
\def\BT{B^{\mathbb{T}_n}}
\def\codim{\text{codim}}
\def\const{c}
\def\CCG{\text{G}}
\def\colim{\text{colim}}
\def\cond{condition }
\def\C{\mbox{\bf C}}
\def\ct{\mathsf{ct}}
\def\d{{\rm d}}
\def\dell{\partial}
\def\diam{\text{diam}}
\def\DSF{\ms{DSF}}
\def\E{\mathbb{E}}
\def\V{\mathbb{V}}
\def\envi{\mathsf{env}}
\def\enviIn{\partial^{\mathsf{in}}}
\def\enviOut{\partial^{\mathsf{out}}}
\def\enviInn{\partial^{\mathsf{in},*}}
\def\enviStab{\mathsf{env}_{\mathsf{stab}}}
\def\Et{\text{Et}}
\def\es{\emptyset}
\def\exp{\text{exp}}
\def\fa{for all }
\def\Fk{\mathcal{F}_{k_0}}
\def\Fm{Furthermore}
\def\G{\mathbb{G}}
\def\gr{\text{gr}}
\def\hge{h_{\mathsf{g}}}
\def\hco{h_{\mathsf{c}}}
\def\H{\text{H}}
\def\Hom{\text{Hom}}
\def\Hs{\widetilde{X}_{H,0}}
\def\inj{\hookrightarrow}
\def\id{\text{id}}
\def\iiets{it is easy to see }
\def\iietc{it is easy to check }
\def\Iietc{It is easy to check }
\def\Iiets{It is easy to see }
\def\imp{\Rightarrow}
\def\({\bih(}
\def\){\big)}
\def\lver{\big|}
\def\rver{\big|}
\def\lcu{\big\{}
\def\rcu{\big\}}
\def\im{\mbox{im}}
\def\inn{\mathsf{in}}
\def\Inv{\text{Inv}}
\def\Ind{\text{Ind}}
\def\Ip{In particular}
\def\ip{in particular }
\def\LB{\text{LB}}
\def\Lo{\mathcal{L}^o}
\def\mc{\mathcal}
\def\ms{\mathsf}
\def\mb{\mathbb}
\def\mf{\mathbf}
\def\Hp{\wt{X}_{H,0}^{'}}
\def\M{\mathbb{M}}
\def\Mo{Moreover}
\def\G{\mathbb{G}}
\def\GR{\mathsf{GR}}
\def\N{\mathbb{N}}
\def\Npo{\mathbf{N}_{\mathcal{P}^o}}
\def\k{\overline{k}}
\def\K{\underline{K}}
\def\LE{\mathsf{LE}}
\def\ldot{.}
\def\lmid{\;\middle\vert\;}
\def\O{\mathcal{O}}
\def\Ob{Observe }
\def\ob{observe }
\def\out{\mathsf{out}}
\def\one{\mathbbmss{1}}
\def\Otoh{On the other hand}
\def\opartial{\partial^{\text{out}}}
\def\ipartial{\partial^{\text{in}}}
\def\eopartial{\partial^{\text{out}}_{\text{ext}}}
\def\eipartial{\partial^{\text{in}}_{\text{ext}}}
\def\p{\prime}
\def\pred{\mathsf{pred}}
\def\Trace{\mathsf{Trace}}
\def\TT{\ms{undef}}
\def\Rout{\TT}
\def\pp{{\prime\prime}}
\def\Po{\mathcal{P}^0}
\def\P{\mathbb{P}}
\def\Proj{\mbox{\bf P}}
\def\Q{\mathbb{Q}}
\def\QQ{\overline{\Q}}
\def\pr{\text{pr}}
\def\R{\mathbb{R}}
\def\rstab{R_{\mathsf{stab}}}
\def\Spec{\text{Spec}}
\def\st{such that }
\def\sl{sufficiently large }
\def\ss{sufficiently small }
\def\sot{so that }
\def\su{suppose }
\def\succ{\mathsf{succ}}
\def\Su{Suppose }
\def\suf{sufficiently }
\def\udot{\mathaccent\cdot\cup}
\def\Set{\mathcal{S}et}
\def\T{\mathbb{T}}
\def\Twh{Then we have }
\def\Tes{There exists }
\def\te{there exist }
\def\tes{there exists }
\def\tptc{this proves the claim}
\def\Map{\text{Map}}
\def\VLo{\mc{VL}^o}
\def\wt{\widetilde}
\def\Wcon{We conclude }
\def\wcon{we conclude }
\def\wc{we compute }
\def\Wc{We compute }
\def\wo{we obtain }
\def\wh{we have }
\def\Wh{We have }
\def\Z{\mathbb{Z}}
\def\ZSlab{\mathbb{Z}^2_L\times\{0\}^{d-2}}
\def\RR{\R^{d}}

\def\XX{{\cal X}}
\def\DD{{\Bbb L}}
\def\Q{{\Bbb Q}}
\def\f{{\varphi}}
\def\rr{{\varrho}}
\def\e{{\varepsilon}}
\def\S{{\Bbb S}}
\def\dev{{\text{dev}}}
\def\cW{{\cal W}}
\def\eps{\varepsilon}
\def\dist{{\rm dist}}
\def\parent{{\rm parent}}
\def\leqd{\stackrel{\cal D}{\preceq}}
\def\ba{{\backslash}}
\def\sb{{\subset}}
\def\ov{\overline}
\def\D{\Delta}
\def\a{\alpha}
\def\sm{\setminus}
\def\ba{\setminus}
\def\b{\beta}
\def\MM{{\cal M}}
\def\EE{{\cal E}}
\def\e{\varepsilon}
\def\f{\varphi}
\def\phi{\varphi}
\def\g{\gamma}
\def\la{\lambda}
\def\k{\kappa}
\def\r{\rho}
\def\de{\delta}
\def\vr{\varrho}
\def\vt{\vartheta}
\def\s{\sigma}
\def\t{{\ttau}}
\def\th{\theta}
\def\x{\xi}
\def\z{\zeta}
\def\o{\omega}
\def\D{\Delta}
\def\L{\Lambda}
\def\G{\Gamma}
\def\O{{\Omega}}
\def\P{{\mathbb{P}}}
\def\SS{{\Sigma}}
\def\T{\T}
\def\Th{{\Theta}}
\def\X{\Xi}
\def\F{{\ms F}}
\def\FF{{\mathcal F}}
\def\HH{{\mathcal H}}
\def\SS{{\cal S}}
\def\AA{{\cal A}}
\def\GG{{\cal G}}
\def\PP{{\cal P}}
\def\NN{{\cal N}}
\def\II{{\cal I}}
\def\MM{\mathcal M_*}
\def\EE{{\mathcal E}}
\def\BB{{\mc{B}}}
\def\BBB{{\mc{K}}}
\def\CCC{{\mc C}}
\def\SS{{\mc S}}
\def\DD{\mc{K}'}
\def\LL{{\mc L}}
\def\Norm{\text{norm}\,}
\def\Const{\text{Const}\,}
\def\const{\text{const}\,}
\def\VVT{{\Vert|}}
\def\V|{{\Vert}}

\def\DDL{{\bf{L}}}
\def\bb{{\bf{b}}}
\def\aa{{\bf{a}}}
\def\cc{{\bf{c}}}
\def\ssm{{\bf{s}}}
\def\nunu{\boldsymbol{\nu}}
\def\ttau{\boldsymbol{\tau}}
\def\nuu{{\bar\nu}}
\def\muu{{\bar\mu}}
\def\nuuu{{\hat\nu}}
\def\muuu{{\hat\mu}}
\def\usc{\text{u.s.c.~}}
\def\lsc{\text{l.s.c.~}}


\newcommand{\blue}[1]{\textcolor{blue}{#1}}
\newcommand{\red}[1]{\textcolor{red}{#1}}
\newcommand{\cyan}[1]{\textcolor{cyan}{#1}}
\newcommand{\yellow}[1]{\textcolor{yellow}{#1}}

\author{Christian Hirsch}
\author{Benedikt Jahnel}
\author{Paul Keeler}
\author{Robert Patterson}
\thanks{Weierstrass Institute Berlin, Mohrenstr. 39, 10117 Berlin, Germany; E-mail: {\tt christian.hirsch@wias-berlin.de}, {\tt benedikt.jahnel@wias-berlin.de}, {\tt paul.keeler@wias-berlin.de}, {\tt robert.patterson@wias-berlin.de}.}

\title{Traffic flow densities in large transport networks}

\date{\today}
\begin{abstract}
We consider transport networks with nodes scattered at random in a large domain. At certain local rates, the nodes generate traffic flowing according to some navigation scheme in a given direction. In the thermodynamic limit of a growing domain, we present an asymptotic formula expressing the local traffic flow density at any given location in the domain in terms of three fundamental characteristics of the underlying network: the spatial intensity of the nodes together with their traffic generation rates, and of the links induced by the navigation. This formula holds for a general class of navigations  satisfying a link-density and a sub-ballisticity condition. As a specific example, we verify these conditions for navigations arising from a directed spanning tree on a Poisson point process with inhomogeneous intensity function.
\end{abstract}


\maketitle

\section{Introduction and main results}
\label{defSec}

For large-scale networks, where many {network nodes simultaneously generate traffic of some kind}, questions of capacity become of central importance. Indeed, the traffic flow in the network can be seriously obstructed if the amount of traffic that needs to be forwarded by a certain node becomes too large. The main results presented here, Theorems~\ref{trafficThmDir} and~\ref{trafficThmRad},  provide asymptotic formulas that allow one to express the traffic density in any given point through fundamental network characteristics: the spatial intensity of users, the traffic intensity and the intensity of links in the network. We envision applications of these results in the study of traffic flow in multi-hop communication and drainage networks.

Based on the general framework developed by Bordenave~\cite{bordNav}, we consider two types of {routing schemes, or, using his terminology, \emph{navigations}.} First, directed navigations, where the flow of transport is steered towards a certain direction. This type of navigation has been studied rigorously in recent years and includes the directed spanning tree~\cite{BB:2007,PenroseDrain1}, Delaunay routing~\cite{delNav}, the directed minimum spanning tree~\cite{bhattDrain,PenroseDrain2} and the Poisson tree~\cite{poisTree}. Second, navigations where the network nodes have a single site as their transport destination. For example, the radial spanning tree introduced by Baccelli and Bordenave~\cite{BB:2007} falls into this second class.

\subsection{Directed networks}\label{DirNet}
In the following we assume without loss of generality that the direction is given by the unit vector $e_1$ for a network on $\RR$. We consider a random network model with a certain intensity profile in the thermodynamic limit. More precisely, let $D$ be an open, bounded and convex subset of $\RR$. Then, for $s\ge1$ let $X^{(s)}$ be a simple point process with intensity function $\lambda^{(s)}:D\to[0,\infty)$, $x\mapsto\lambda^{(s)}(x)=\lambda(x/s)$ for some continuous and bounded mapping $\lambda: D\to[0,\infty)$. Each node $X_i\in X^{(s)}$ generates traffic at rate $\mu^{(s)}(X_i)=\mu(X_i/s)$, where $\mu: D\to[0,\infty)$ is also a continuous and bounded function. 

Similar to the directed navigation scheme proposed by Bordenave~\cite{bordNav}, we think of a navigation as a function mapping $X^{(s)}$ to itself, where the function value of a node will also be referred to as the \emph{successor} of that node. {In other words, a navigation makes formal the concept of a routing scheme in the current setting.} Additionally, we impose the property that successors always lie to the right. Let $\pi_1:\R^d\to\R$ denote the projection to the first coordinate and $\mathbf{N}$ the set of all locally finite sets of points in $\R^d$.

\begin{definition}
\label{routeDef}
A measurable function $\mc{A}:\RR\times\mathbf{N}\to\RR$ is called \emph{directed navigation} if the following properties are satisfied.
\begin{enumerate}
\item $\mc{A}(x,\varphi)\in\varphi$, for all $x\in\varphi$,
\item $\pi_1(\mc{A}(x,\varphi)) \ge\pi_1(x)$ for all $x\in\varphi$.
\end{enumerate}
\end{definition}
If there is no danger of ambiguity, we will often write $\mc{A}(x)$ instead of $\mc{A}(x,\varphi)$. 
In Figure~\ref{DirFig} we present two examples of directed navigation.
\begin{figure}[!htpb]
\centering
\input{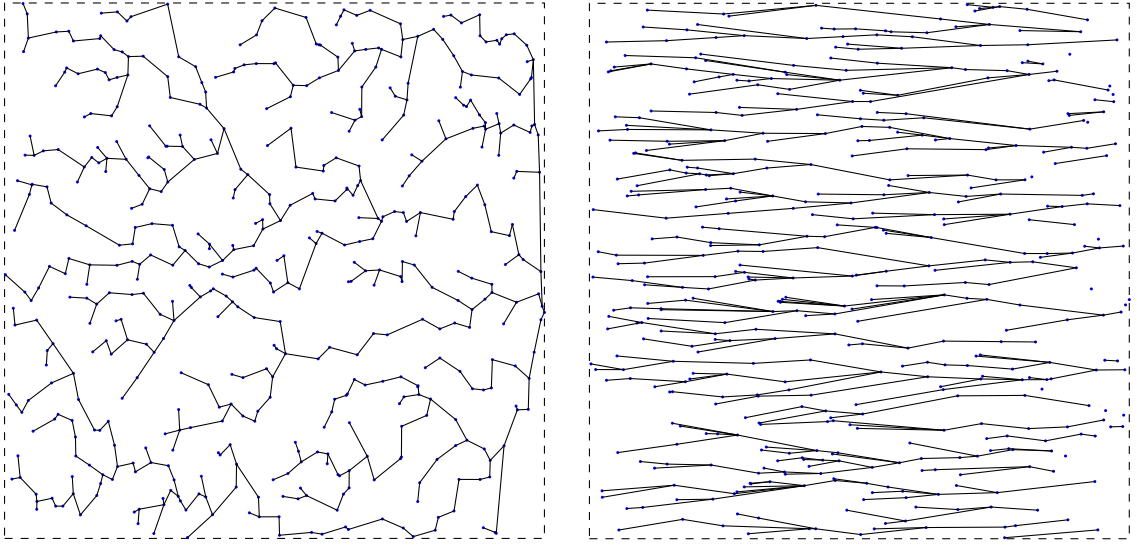}
\caption{Directed navigation based on a Poisson point process on the unit square. Each node connects to its nearest neighbor which is also contained in a wide, respectively narrow, horizontal cone starting at the node.}
\label{DirFig}
\end{figure}


Suppose now that the nodes $X^{(s)}$ generate traffic at rates given by the function $\mu^{(s)}$ and that the traffic is forwarded in the $e_1$-direction according to the navigation $\mc{A}$. Then, the traffic flow at a node $X_i$ can be thought of as the sum of all rates associated with nodes whose route passes through $X_i$. To be more precise, define the $k$-fold iteration $\mc{A}^k$ of $\mc{A}$ recursively by putting $\mc{A}^k(x,\varphi)=\mc{A}^{k-1}(\mc{A}(x,\varphi),\varphi)$, $k\ge1$ and $\mc{A}^0(x,\varphi)=x$. Then we can trace the path of the traffic originating from a node $X_i\in X^{(s)}$ by considering the \emph{trajectory} $\Gamma(X_i,X^{(s)})=\{\mc{A}^k(X_i,X^{(s)}):\,k\ge0\}$. In other words, $\Gamma(X_i)=\Gamma(X_i,X^{(s)})$ consists of all the iterated successors of $X_i$. Let us also define $\bar\Gamma(X_i)=\bigcup_{k\ge0}[\mc{A}^{k}(X_i),\mc{A}^{k+1}(X_i)]$ as the \emph{interpolated trajectory}.
We define the \emph{traffic flow} at $X_i\in X^{(s)}$ as
$$\Delta(X_i)=\Delta(X_i,X^{(s)})=\sum_{X_j:\, X_i\in\Gamma(X_j)}\mu^{(s)}(X_j).$$
In the following, we analyze the asymptotic behavior of the spatial traffic flow density at any given location as $s$ tends to infinity. Here the spatial traffic flow density is understood as a spatial average of traffic flow in a certain \emph{microscopic environment}, whose diameter is of order $o(s)$. In order to get a tractable limit result, we need to impose some restrictions on the navigation $\mc{A}$.

\medskip

For a point $x\in\RR$ let $B^{d}_r(x)$ denote the $d$-dimensional ball with radius $r$ around $x$. First, we need to specify a suitable notion of density of links induced by the navigation $\mc{A}$ passing through an $g(s)$-neighborhood of $sx$, where $g:[1,\infty)\to[1,\infty)$ is an unbounded increasing function with $g(s)/s$ tending {monotonically} to zero as $s$ tends to infinity. 
More precisely, let
$$I^{\mathsf{D}}_{s}(x)=B^d_{g(s)}(sx)\cap\{y\in\RR:\, \pi_1(y)=\pi_1(sx)\}$$
denote the environment of $x$ inside the hyperplane through $x$ perpendicular to $e_1$. 
Our first condition ensures the existence of the asymptotic intensity of the point process
 $$\Xi^{\mathsf{D}}_{s}(x)=\{X_i\in X^{(s)}:[X_i,\mc{A}(X_i)]\cap\, I^{\mathsf{D}}_{s}(x)\ne\es\}$$
consisting of those points $X_i\in X^{(s)}$ such that the segment $[X_i,\mc{A}(X_i)]$ crosses $I^{\mathsf{D}}_{s}(x)$. More precisely, we define the following \textit{link-density condition}, where $\nu_{d-1}$ denotes the $(d-1)$-dimensional Hausdorff measure in $\RR$.

\begin{enumerate}
\item[{\bf (D1)}] There exists a function $\lambda_{\mc{A}}:D\to(0,\infty)$ such that for every $x\in D$, 
$$\lambda_{\mc{A}}(x)=\lim_{s\to\infty}\nu_{d-1}(I^{\mathsf{D}}_{s}(x))^{-1}\E\#\Xi^{\mathsf{D}}_{s}(x).$$
\end{enumerate}
In words, $\lambda_{\mc{A}}(x)$ denotes the intensity of network links crossing the surface $I^{\mathsf{D}}_{s}(x)$.

\medskip
Second, we assume that the paths $\Gamma(X_i)$, for all $X_i\in X^{(s)}$ satisfy a certain sub-ballisticity condition{, to be defined shortly,} with high probability. Loosely speaking, when considering any trajectory $\Gamma(X_i)$ away from the boundary of $sD$, then the deviation from the horizontal ray starting from $X_i$ should be of order $O(h(s))$. Here $h:[1,\infty)\to[1,\infty)$ is an unbounded increasing function with $h(s)/s$ tending {monotonically} to zero as $s$ tends to infinity.
More precisely, we write 
$$D_\e=\{x\in D:\, B_{\e}(x)\subset D\}$$ 
for the points of distance  at least $\e$ to the boundary of $D$. Moreover, we denote by
$$Z_r^{\ms{D}}(x)=\R\times B^{d-1}_{r}(o)+x$$
the horizontal cylinder with radius $r$ shifted to the point $x\in\RR$.
We define
$$E^{\mathsf{D}}_{s,\e}=E^{\mathsf{D}}_{s,\e,1}\cap\, E^{\mathsf{D}}_{s,\e,2}$$ where
\begin{align*}
E^{\mathsf{D}}_{s,\e,1}=\{\mc{A}(X_i)\neq X_i\text{ for all }X_i\in X^{(s)}\cap\,  (sD)_{\e s}\}
\end{align*}
is the event that, away from the boundary of $sD$, there are no \textit{dead ends}, that is, no points $X_i\in X^{(s)}$ satisfying $\mc{A}(X_i)=X_i$, and 
\begin{align*}
	E^{\mathsf{D}}_{s,\e,2}=\{(\bar\Gamma(X_i)\cap\, (sD)_{\e s})\subset Z^{\ms{D}}_{h(s)}(X_i)\text{ for all }X_i\in X^{(s)}\}
\end{align*}
is the event that, away from the boundary of $sD$, all interpolated trajectories remain in the $h(s)$-cylinder centered at their starting points. Let us now define the following \textit{sub-ballisticity condition}.
\begin{enumerate}
\item[{\bf (D2)}\label{D2}] 
For all $\e>0$, we have $1-\P(E^{\mathsf{D}}_{s,\e})\in O(s^{-2d})$. 
\end{enumerate}
Note that this sub-ballisticity condition can be seen as a finite analog of the straightness condition introduced by Howard and Newman~\cite{efpp2}. 

\medskip
Now we come to the main results for directed navigation, where formally $\la$ and $\mu$ are extended by zero outside of $D$. 

\begin{theorem}
\label{trafficThmDir}
Let $x\in D$ be arbitrary. Assume that conditions \emph{{\bf(D1)}} and \emph{{\bf(D2)}} are satisfied with $h(s)\in o(g(s))$ and that $\E[(\#X^{(s)})^2]\in O(s^{2d})$. 
\begin{enumerate}
\item Then, 
\begin{align}
\label{kappaFormDir}
\lim_{s\to\infty}s^{-1}\frac{\E\sum_{X_i\in \Xi^{\mathsf{D}}_{s}(sx)}\Delta(X_i)}{\E\# \Xi^{\mathsf{D}}_{s}(sx)}=\lambda_{\mc{A}}(x)^{-1}\int^{0}_{-\infty}\lambda(x+re_1)\mu(x+re_1)\d r.
\end{align}
\item If, additionally, $X^{(s)}$ is either a Poisson point process or  $\mu$ is constant on $D$ and $X^{(s)}=X\cap\, sD$ for some ergodic point process $X$, then 
\begin{align}
\label{kappaFormDir2}
\lim_{s\to\infty}s^{-1}\frac{\sum_{X_i\in \Xi^{\mathsf{D}}_{s}(sx)}\Delta(X_i)}{\E\# \Xi^{\mathsf{D}}_{s}(sx)}=\lambda_{\mc{A}}(x)^{-1}\int^{0}_{-\infty}\lambda(x+re_1)\mu(x+re_1)\d r
\end{align}
in probability.
\end{enumerate}
\end{theorem}
As an example, in Section~\ref{Verification}, we illustrate how to prove the conditions in the case of the directed spanning tree.
Let us note that we only need to know existence of fluctuation functions $g$ and $h$, but not their precise form in order to compute the right-hand side (r.h.s.)~of the above equations. Before providing a rigorous proof of Theorem~\ref{trafficThmDir}, let us give some intuition for the expressions appearing in~\eqref{kappaFormDir} and~\eqref{kappaFormDir2}. First, as mentioned above, the left-hand side (l.h.s.)~can be interpreted as the rescaled traffic flows that are averaged over all nodes in a microscopic environment around $sx$. In particular, the local traffic flow density grows asymptotically linearly in $s$ where the leading order coefficient is given by the r.h.s. For this coefficient, the sub-ballisticity condition guarantees, that in the limit, only nodes inside a narrowing horizontal tube contribute to the traffic flow. This is reflected by the fact that in the r.h.s.~the integration is performed over the interval from $-\infty$ to $0$, accumulating the spatial intensities of sites and the associated traffic along the horizontal line.

\subsection{Radial networks}
In contrast to directed networks, here we consider navigations that create trees with the origin $o$ as their center. Similar to the setup for directed networks, we will study a large network of nodes. As before, let $D$ be an open, bounded and convex subset of $\R^d$, which now includes the origin.
The simple point process $X^{(s)}$, $s\ge1$, has intensity $x\mapsto\lambda^{(s)}(x)=\lambda(x/s)$, where $\lambda: D\to[0,\infty)$ is continuous and bounded. Additionally, each node $X_i\in X^{(s)}$ generates traffic at rate $\mu^{(s)}(X_i)=\mu(X_i/s)$, where $\mu: D\to[0,\infty)$ is some continuous and bounded function. 
Similar to navigations proposed in~\cite{bordNav}, we define a radial navigation scheme as follows.
\begin{definition}
\label{radRouteDef}
A measurable function $\mc{A}:\R^d\times\mathbf{N}\to\R^d$ is called \emph{radial navigation} if the following properties are satisfied.
\begin{enumerate}
\item $\mc{A}(o,\varphi\cup\{o\})=o$,
\item $\mc{A}(x,\varphi)\in\varphi\cup\{o\}$ for all $x\in\varphi$,
\item $\#\mc{A}(x,\varphi)<|x|$ for all $x\in\varphi\setminus \{o\}$.
\end{enumerate}
\end{definition}
In Figure~\ref{RadFig}, we illustrate a radial navigation based on nodes distributed according to a Poisson point process.
We note that in~\cite{bordNav}, item (iii) is not part of the definition. Nevertheless, navigations with this property are considered as the special class of navigations with positive progress. Trajectories  and traffic flow are defined as in the directed case.
Again we have to introduce two conditions under which we prove an asymptotic traffic flow density result. 

\begin{figure}[!htpb]
\centering
\input{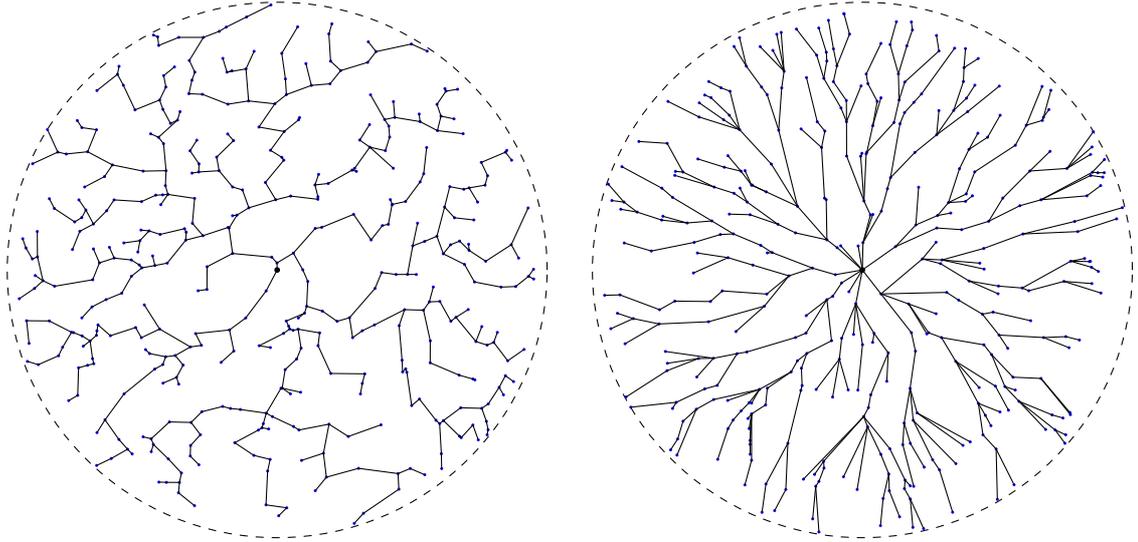}
\caption{Radial navigation based on a Poisson point process on the disc. Each node connects to its nearest neighbor which is also closer to the origin, respectively additionally is contained in a cone starting at the node and opening towards the origin.}
\label{RadFig}  
\end{figure}

\medskip
First, for the link-density condition \emph{{\bf(D1)}}, we need to adapt the definitions of $I^{\mathsf{D}}$ and $\Xi^{\mathsf{D}}$ to the radial network setting. More precisely,  we fix two unbounded increasing functions $g,h:[1,\infty)\to[1,\infty)$ such that $g(s)/s$ and $h(s)/s$ tend monotonically to zero as $s$ tends to infinity. Then, we put 
$$I^{\mathsf{R}}_{s}(x)=B^d_{g(s)}(sx)\cap\partial B^d_{s|x|}(o),$$
noting that the hyperplane surface in case of directed networks here becomes a spherical cap.  
Moreover, we put
$$\Xi^{\mathsf{R}}_{s}(x)=\{X_i\in X^{(s)}:[X_i,\mc{A}(X_i)]\cap\, I^{\mathsf{R}}_{s}(x)\ne\es\}$$
for those points $X_i\in X^{(s)}$ such that the segment $[X_i,\mc{A}(X_i)]$ crosses $I^{\mathsf{R}}_{s}(x)$. 
\begin{enumerate}
\item[{\bf (R1)}] There exists a function $\lambda_{\mc{A}}: D\to[0,\infty)$ such that for every $x\in D\setminus\{o\}$, 
$$\lambda_{\mc{A}}(x)=\lim_{s\to\infty}\nu_{d-1}(I^{\mathsf{R}}_{s}(x))^{-1}\E\#\Xi^{\mathsf{R}}_{s}(x).$$
\end{enumerate}

\medskip
Second, we change the sub-ballisticity condition \emph{{\bf(D2)}} in such a way that the cylinders point towards the origin. More precisely, we write $\hat{v}=v/|v|$ for $v\in\R^d\setminus\{o\}$ and define
$$Z^{\ms{R}}_{r}(v)=\{y\in \R^d: |y-\langle y,\hat v\rangle \hat v|\le r\}$$ 
for the cylinder consisting of all points in $\R^d$ whose projection onto the orthogonal complement of the direction $\hat v\in\partial B^d_1(o)$ is of length at most $r\ge0$. Moreover,
$$E^{\mathsf{R}}_{s}=\{\g(X_i)\subset Z^{\ms{R}}_{h(s)}(X_i) \text{ for all }X_i\in X^{(s)}\}$$
is the event that the trajectory is contained in a narrow cylinder directed towards the origin.
\begin{enumerate}
\item[{\bf (R2)}\label{R2}] $1-\P(E^{\mathsf{R}}_{s})\in O(s^{-2d})$. 
\end{enumerate}
Observe that in the definition of {\bf (R2)}, in comparison to {\bf (D2)}, there is no restriction regarding the boundary of $D$. This comes from the fact that, in the radial case, trajectories are always directed towards the center and therefore cannot enter the $\e$-boundary from the $\e$-interior. On the contrary, in the directed case, if a trajectories enters the right $\e$-boundary of $D$ it is pushed further in the direction of that boundary and therefore cannot be controlled. 

\medskip
Now we come to our second main result. 
\begin{theorem}
\label{trafficThmRad}
Let $x\in D\setminus\{o\}$ be arbitrary. Assume that conditions \emph{{\bf(R1)}} and \emph{{\bf(R2)}} are satisfied with $h(s)\in o(g(s))$ and that $\E[(\#X^{(s)})^2]\in O(s^{2d})$. 
\begin{enumerate}
\item Then, 
\begin{align}
\label{kappaFormRad}
\lim_{s\to\infty}s^{-1}\frac{\E\sum_{X_i\in \Xi^{\mathsf{R}}_{s}(x)}\Delta(X_i)}{\E\# \Xi^{\mathsf{R}}_{s}(x)}=|x|^{-d+1}\lambda_{\mc{A}}(x)^{-1}\int_{|x|}^{\infty}\lambda(r\hat{x})\mu(r\hat{x})r^{d-1}\d r.
\end{align}
\item If, additionally, $X^{(s)}$ is either a Poisson point process or  $\mu$ is constant on $D$ and $X^{(s)}=X\cap\, sD$ for some ergodic point process $X$, then 
\begin{align}
\label{kappaFormRad2}
\lim_{s\to\infty}s^{-1}\frac{\sum_{X_i\in \Xi^{\mathsf{R}}_{s}(x)}\Delta(X_i)}{\E\# \Xi^{\mathsf{R}}_{s}(x)}=|x|^{-d+1}\lambda_{\mc{A}}(x)^{-1}\int_{|x|}^{\infty}\lambda(r\hat{x})\mu(r\hat{x})r^{d-1}\d r
\end{align}
in probability.
\end{enumerate}
\end{theorem}

Note that the additional factors $r^{d-1}$ and $|x|^{-d+1}$, which are not present in the directed case, correspond to the fact that in the radial network, the limiting integration domain is given by a narrowing conical frustum, with a radius depending on $|x|$ and not, as in the directed case, by a narrowing tube.

Theorem~\ref{trafficThmRad} is an extension of earlier results in the engineering literature~\cite{mao,rodolakis}, which provide formulas for the average traffic flow passing through a macroscopic shell around the origin. Our result gives the asymptotically precise value of the traffic-flow density averaged in a microscopic environment around any given point in an inhomogeneous network.

\subsection{Verification of (D1) and (D2) for inhomogeneous processes}\label{Verification}
In order to illustrate the applicability of our main results, we verify the abstract conditions for a standard example of a directed navigation, namely the directed spanning tree considered in~\cite{BB:2007,bordNav}. This navigation is defined as follows. If $\varphi\subset\R^d$ is locally finite and $x\in\varphi$, then $\mc{A}(x,\varphi)$ is defined to be the point in $\varphi\cap ((\pi_1(x),\infty)\times\R^{d-1})$ that is of minimal Euclidean distance to $x$. If the minimum is realized in several points, we choose the lexicographic smallest of them. If $x$ is the right-most point of $\varphi$, then we put $\mc{A}(x,\varphi)=x$.

Starting from a homogeneous Poisson point process, it would not be difficult to deduce conditions \emph{{\bf(D1)}} and \emph{{\bf(D2)}} from the results of~\cite{BB:2007,bordNav}. However, one of the strengths of our asymptotic traffic density formula is its validity under rather general assumptions on the underlying network, including, in particular, situations with inhomogeneous node distributions. Therefore,  we take a further step and show that these conditions remain valid if we replace the constant intensity by a general inhomogeneous intensity function that is subject only to a Lipschitz constraint. From now on, $\mc{A}$ denotes the directed-spanning tree navigation. 

The link-density and the sub-ballisticity conditions are of fundamentally different nature. The link-density condition is often easier to verify  since it only depends on the local behavior of the navigation. 

\begin{proposition}\label{LinkDensity}
	Let $\la$ be as in Section~\ref{DirNet} and locally Lipschitz. Further, let $X^{(s)}$ be a Poisson point process on $sD$ with intensity function $\la^{(s)}$.
Then, for every $0<\xi<1$ condition \emph{{\bf(D1)}} is satisfied for the directed spanning tree on $X^{(s)}$ together with the fluctuation function $g(s)=s^{\xi}$.
\end{proposition}

\medskip
Let us comment also on the sub-ballisticity condition {\bf(D2)}. In~\cite[Theorem 4.10]{BB:2007} a closely related property is proved for the directed spanning tree constructed on the homogeneous Poisson point process.
Moreover, the fluctuation function is given by $s^{1/2+\e}$ which indicates the diffusive character of the scaling. Using this result, we can verify condition {\bf(D2)} also for inhomogeneous Poisson point processes, provided that the intensity function is Lipschitz. 

As a simplification, we assume that the convex domain $D$ is given by the unit cube in $\R^d$.
\begin{proposition}\label{InhomBalli}
	Let $\la$ be as in Section~\ref{DirNet} and assume additionally that $D=[-1/2,1/2]^d$ and that $\la$ is Lipschitz. Furthermore, let $X^{(s)}$ be a Poisson point process on $sD$ with intensity function $\la^{(s)}$. Then, there exists $0<\xi<1$ such that condition \emph{{\bf(D2)}} is satisfied for the directed spanning tree on $X^{(s)}$ together with the fluctuation function $h(s)=s^{\xi}$.
\end{proposition}
In fact, in~\cite{bordNav} sub-ballisticity is not only checked for the specific choice of the directed spanning tree, but more generally for a class of regenerative navigations with certain additional properties. Similarly, Proposition~\ref{InhomBalli} could be extended in this direction. However, to keep the presentation accessible, we restrict our attention to the important special case of the directed spanning tree on inhomogeneous Poisson point processes.

\subsection{Conjectures for navigation schemes based on bounded range radii}
{In the setting of wireless networks}, much of the work by practitioners on routing algorithms has been done under the assumption that the ranges or radii of the transmitters are bounded, which has led to a number of routing schemes being proposed and studied~\cite{keeler2010stochastic,keeler2012random,takagi1984optimal}. 

In a radial setting the assumption of bounded radii has important consequences as it implies that dead ends can occur. More precisely, a positive proportion of nodes is isolated in the sense that there is no node within the range that lies closer to the origin. Since this violates condition (iii) in Definition~\ref{radRouteDef}, this condition must be replaced by condition (iii-a) where $<$ is replaced by $\le$. More importantly, simply adding a radius constraint in the radial spanning tree results in a highly disconnected network, where only a small number of nodes can communicate with the origin before reaching a dead end. 

Nevertheless, by implementing a global navigation algorithm, it is possible to build a working transport network with bounded ranges. More precisely, we may consider all paths from a given node to the origin where every step is closer to the origin and additionally the constraint of bounded ranges is not violated. If such a path does not exist, the node is considered to be a dead end. In the other case, we may choose a path with a minimum number of hops. The collection of all chosen paths gives rise to a navigation with condition (iii-a).

Is it possible, under link-density and sub-ballisticity conditions, to give a generalized version of Theorem~\ref{trafficThmRad} for navigations with condition (iii-a)? Is it reasonable to believe that the conditions are valid for standard examples in the class of navigations with condition (iii-a)?
Regarding the first question, it is plausible that under Poisson assumptions the formula for the asymptotic mean-value in Theorem~\ref{trafficThmRad} only needs to be extended by an additional factor in the integrand taking into account the probability of getting stuck in a dead end. 

Regarding the second question, the link-density and sub-ballisticity conditions should be satisfied for minimum-hop navigations. Indeed, in two dimensions, it is expected that in the thermodynamic limit, trajectories converge to the uniquely determined semi-infinite geodesic with a certain direction~\cite{efpp,licea}. Moreover, again in two dimensions, it is conjectured that many first-passage percolation models are sub-ballistic with exponent $2/3$, see~\cite{kestenFest,piza}. However, up to now, partial results related to the link-density and sub-ballisticity condition have only been obtained under restrictive assumptions such as isotropy and planarity. This makes it rather difficult to provide a non-artificial type of bounded-range navigation where the two central conditions can be checked rigorously. 
\input{Thm}

\input{Prop}

\section*{Acknowledgments}
This research was supported by the Leibniz program \emph{Probabilistic Methods for Mobile Ad-Hoc Networks}. 
\bibliography{../wias}
\bibliographystyle{abbrv}
\end{document}

%% file: Thm.tex
\section{Proofs of Theorem \ref{trafficThmDir} and \ref{trafficThmRad} }
Let us denote $\mu_{\ms{max}}=\sup_{x\in D}\mu(x)$ and $\la_{\ms{max}}=\sup_{x\in D}\la(x)$ where $\mu_{\ms{max}}<\infty$ and $\la_{\ms{max}}<\infty$ by the boundedness assumptions. We will also write $A^c$ for the complement of the set $A$. 

\subsection{Proof of Theorem \ref{trafficThmDir}}
Let us start with two lemmas giving estimates of the accumulated traffic flow under the events $E_{s,\e}^{\mathsf{D}}$ and $(E_{s,\e}^{\mathsf{D}})^c$, respectively. For this we need the following definitions. We write $R^+_{s}(x)=Z^{\ms{D}}_{g(s)+h(s)}(sx)$ and 
$$R^{\ms{left},+}_s(x)=\{sy\in s D\cap  R^+_s(x):\, \pi_1(y)\le\pi_1(x)\}$$
for the set of points in $sD$ which lie in the microscopic cylinder $R^+_s(x)$ to the left of $sx$. Further we write $R^-_{s}(x)=Z^{\ms{D}}_{g(s)-h(s)}(sx)$ and 
$$R^{\ms{left},-}_{s,\e}(x)=\{sy\in s D\cap R^-_s(x):\,\zeta^-_s(\e)\le \pi_1(y)\le\pi_1(x)\}$$ 
where 
$\zeta^-_s(\e)=\inf\{\pi_1(z):\, z\in D_\e\cap\, (R^-_s(x)/s)\}$
denotes the smallest first coordinate of points which are in the $\e$-interior of $D$ intersected with the microscopic cylinder $R^-_s(x)/s$, see also Figure~\ref{cylFig}.

\begin{figure}[!htpb]
\centering
\begin{tikzpicture}[scale=1.0]

\begin{scope}
\clip (4,1.1) rectangle (-4,-1.1);
\fill[black!50!white] plot [smooth,thick,tension=1.0] coordinates {(-2,-2) (-1,1) (4,2)} (4,2)--(4,-2)--(-2,-2);
\end{scope}

\fill[black!20!white] (-0.56,-0.56) rectangle (4,0.56);
\fill[black!100!white] (3.98,-0.83) rectangle (4.02,0.83);

\draw plot [smooth,thick,tension=1.0] coordinates {(-2,-2) (-1,1) (4,2)};
\draw plot [smooth,thick,tension=1.0] coordinates {(-1.3,-2.0) (-0.5,0.6) (4.0,1.45)};

\coordinate[label=-135: {${sx}$}] (sx) at (4,0.0);
\coordinate (sxe) at (-1.11,0);
\coordinate[label=90: {${sx_{\e}}$},xshift=-0.18cm] (sxee) at (-1.18,0);
\coordinate[label={[label distance=-0.1cm]135: {${sx^*}$}},xshift=-1.0cm] (sxs) at (-0.88,0);
\coordinate[label={[label distance=-0.1cm]-45: {${sz}$}},yshift=-0.0cm] (zse) at (-0.56,0.56);
\coordinate[label={0: {${\partial (sD)}$}},yshift=-0.0cm] (sD) at (4,2);
\coordinate[label={0: {${\partial((sD)_{s\e})}$}},yshift=-0.0cm] (sDe) at (4,1.5);

\draw[decorate,decoration={brace,mirror}] (4.1,-0.83)--(4.1,0.83);
\draw[decorate,decoration={brace,mirror}] (4.1,0.83)--(4.1,1.1);

\draw (sx)--(sxs);

\coordinate[label=0: {$I^{\mathsf{D}}_{s}(x)$},xshift=0.08cm] (hg) at (4.1,0.0);
\coordinate[label=0: {$h(s)$},xshift=0.1cm] (hh) at (4.1,0.9);

\foreach \i in {sx,sxe,sxs,zse}
{
\fill (\i) circle (1pt);
}

\end{tikzpicture}
\caption{Construction of the cylinders $R^{\ms{left},-}_{s,\varepsilon}(x)$ (light gray) and $R^{\ms{left},+}_s(x)$ (union of light and dark gray) where $\pi_1(z)=\zeta_s^-(\e)$.}
\label{cylFig}
\end{figure}

\begin{lemma}\label{Thm2Lemma1}
Let $\e>0$ and $x\in D_{2\e}$ be arbitrary. Further, let  $X^{(s)}\in E_{s,\e}^{\mathsf{D}}$ then
\begin{align*}
	\sum_{X_j\in X^{(s)}\cap\, R^{\ms{left},-}_{s,\e}(x)}\mu^{(s)}(X_j)\le \sum_{X_i\in X^{(s)}\cap\, \Xi^{\mathsf{D}}_{s}(x)}\Delta(X_i)\le\sum_{X_j\in X^{(s)}\cap\, R^{\ms{left},+}_s(x)}\mu^{(s)}(X_j).
\end{align*}
\end{lemma}
\begin{proof}
For the upper bound, the cylinder condition given by $E_{s,\e}^{\mathsf{D}}$ ensures that we can estimate
\begin{align*}
\sum_{X_i\in X^{(s)}\cap\,\Xi^{\mathsf{D}}_{s}(x)}\Delta(X_i)
&=\sum_{X_j\in X^{(s)}}\mu^{(s)}(X_j)\sum_{X_i\in \Xi^{\mathsf{D}}_{s}(x)}\one_{\g(X_j)}(X_i)\\
&=\sum_{X_j\in X^{(s)}\cap\, R^{\ms{left},+}_s(x)}\mu^{(s)}(X_j)\sum_{X_i\in\,  \Xi^{\mathsf{D}}_{s}(x)}\one_{\g(X_j)}(X_i)\le\sum_{X_j\in X^{(s)}\cap\, R^{\ms{left},+}_s(x)}\mu^{(s)}(X_j).
\end{align*}
For the lower bound, since $E_{s,\e}^{\mathsf{D}}$ does not give control over nodes in $sD\setminus (sD)_{s\e}$, those nodes have to be excluded. 
\end{proof}

\begin{lemma}\label{Thm2Lemma2} $\E[\one_{(E_{s,\e}^{\mathsf{D}})^c}\sum_{X_i\in \Xi^{\mathsf{D}}_{s}(x)}\Delta(X_i)]\in O(1)$.
\end{lemma}
\begin{proof}
Note that 
\begin{align*}
\sum_{X_i\in \Xi^{\mathsf{D}}_{s}(x)}\Delta(X_i)&
	=\sum_{X_j\in X^{(s)}}\mu^{(s)}(X_j)\sum_{X_i\in\, \Xi^{\mathsf{D}}_{s}(x)}\one_{\g(X_j)}(X_i)
\le \mu_{\ms{max}}\# X^{(s)}
\end{align*}
and hence using Cauchy-Schwarz,
\begin{align*}
	\E\one_{(E_{s,\e}^{\mathsf{D}})^c}\sum_{X_i\in\, \Xi^{\mathsf{D}}_{s}(x)}\Delta(X_i)\le\mu_{\ms{max}}\sqrt{1-\P(E_s^{\mathsf{D}})}\sqrt{\E[(\# X^{(s)})^2]}\in O(1).&\qedhere
\end{align*}
\end{proof}
Now, we are in the position to prove Theorem \ref{trafficThmDir}. For convenience let us abbreviate $i_s=s\nu_{d-1}(I^{\mathsf{D}}_{s}(x))$, $N_s=i_s^{-1}\sum_{X_i\in X^{(s)}\cap\Xi^{\mathsf{D}}_s(x)}\Delta(X_i)$ and $S= \int^{0}_{-\infty}\lambda(x+re_1)\mu(x+re_1)\d r$.
\begin{proof}[Proof of Theorem \ref{trafficThmDir}]
\textbf{Part (i): }By condition {\bf{(D1)}} it suffices to show
	$\lim_{s\to\infty}\E N_s=S$.
Using Lemma \ref{Thm2Lemma1}, Lemma \ref{Thm2Lemma2}, Campbell's theorem for random sums and coordinate transformation we can estimate for the upper bound
\begin{align*}
	\E N_s&=\E\one_{E_{s,\e}^{\mathsf{D}}} N_s+\E\one_{(E_{s,\e}^{\mathsf{D}})^c}N_s\cr
	&\le i_s^{-1}\E\sum_{X_j\in X^{(s)}\cap\, R^{\ms{left},+}_s(x)}\mu^{(s)}(X_j)+o(1)\cr
	&=i_s^{-1}\int_{R^{\ms{left},+}_s(x)}\mu^{(s)}(y)\la^{(s)}(y)\d y+o(1)\cr
&=s^di_s^{-1}\int_{R^+_s(x)/s}\one_{\pi_1(y)\le\pi_1(x)}\mu(y)\la(y)\d y+o(1).
\end{align*}
Note that $y\in R^+_s(x)/s$ if and only if $\bar y=(y_2,\dots,y_d)\in B^{d-1}_{(h(s)+g(s))/s}(x_2,\dots,x_d)$. Hence, by Fubini's Theorem, we have 
\begin{align*}
	\E N_s&\leq s^{d-1}\nu_{d-1}(B^{d-1}_{g(s)}(o))^{-1}\int_{-\infty}^{0}\int_{B^{d-1}_{(h(s)+g(s))/s}(o)}\mu(x+re_1+\bar y)\la(x+re_1+\bar y)\d \bar y\d r+ o(1).
\end{align*}
Since $\mu$ and $\la$ are continuous, the dominated convergence theorem implies that the last line converges to $S$, as required. 

\medskip
As for the lower bound, assume $s'$ to be large enough such that $g(s)-h(s)>0$ for all $s\ge s'$. Then, using again Lemma \ref{Thm2Lemma1} we can estimate
\begin{align*}
\E\one_{E_{s,\e}^{\mathsf{D}}}\sum_{X_i\in\, \Xi^{\mathsf{D}}_{s}(x)}\Delta(X_i)
&\ge\E\sum_{X_j\in X^{(s)}\cap\, R^{\ms{left},-}_{s,\e}(x)}\mu^{(s)}(X_j)-\E\one_{(E_{s,\e}^{\mathsf{D}})^c}\sum_{X_i\in\, \Xi^{\mathsf{D}}_{s}(x)}\Delta(X_i)\end{align*}
where the second summand is in $O(1)$ by Lemma \ref{Thm2Lemma2} and hence can be neglected in the scaling $i_s$.
Again using Campbell's theorem and coordinate transformation, for the first summand we can write
\begin{align*}
	\E&\sum_{X_j\in X^{(s)}\cap\, R^{\ms{left},-}_{s,\e}(x)}\mu^{(s)}(X_j)=\int_{R^{\ms{left},-}_{s,\e}(x)}\mu^{(s)}(y)\la^{(s)}(y)\d y\cr
&=s^d\int_{R^-_s(x)/s}\one_{\pi_1(y)\le\pi_1(x)}\mu(y)\la(y)\d y
-s^d\int_{R^-_s(x)/s}\one_{\pi_1(y)<\zeta^-_s(\e)}\mu(y)\la(y)\d y,
\end{align*}
where, as above $y\in R^-_s(x)/s$ if and only if $\bar y=(y_2,\dots,y_d)\in B^{d-1}_{(g(s)-h(s))/s}(x_2,\dots,x_d)$. Hence, by Fubini's Theorem, the rescaled first summand converges to $S$ as required. 
As for the second summand,
\begin{align}\label{LeftEstimate}
	\frac{s^{d-1}}{\nu_{d-1}(B^{d-1}_{g(s)}(o))}&\int_{R^-_s(x)/s}\hspace{-0.3cm}\one_{\pi_1(y)<\zeta^-_s(\e)}\mu(y)\la(y)\d y\le \mu_{\ms{max}}\la_{\ms{max}}\frac{\nu_{d-1}(B^{d-1}_{g(s)-h(s)}(o))}{\nu_{d-1}(B^{d-1}_{g(s)}(o))}(\zeta^-_s(\e)-\zeta^-_s),
\end{align}
where
$$\zeta^-_s=\inf\{\pi_1(z):\,z\in \partial D\cap\, (R^-_s(x)/s)\}$$ 
denotes the smallest first coordinate of points on the boundary of $D$ intersected with the cylinder $R^-_s(x)/s$.
In order to conclude that the r.h.s.~of \eqref{LeftEstimate} tends to zero, first note that $\lim_{s\to\infty}\nu_{d-1}(B^{d-1}_{g(s)-h(s)}(o))/\nu_{d-1}(B^{d-1}_{g(s)}(o))=1$. Second,  let $x^-$ denote the unique intersection of the negative horizontal ray $x-(0,\infty)e_1$ with $\partial D$, the boundary of $D$, where uniqueness is a consequence of the convexity of $D$. Further, let $x^-(\e)$ denote the unique intersection of the ray $x-(0,\infty)e_1$ with $\partial D_\e$. Now, for $s$ tending to infinity, $\zeta_s^-$ tends to $\pi_1(x^-)$, $\zeta^-_s(\e)$ tends to $\pi_1(x^-(\e))$ and
it suffices to show that $\lim_{\e\to0}x^-(\e)=x^-$. But this is the case since $\pi_1(x^-({\e_1}))\le\pi_1(x^-({\e_2}))$ for $\e_{1}\le\e_{2}$ and if $\lim_{\e\to0}x^-(\e)=x'\neq x^-$ then, since $x'\in\partial D$, the uniqueness of the intersection would be violated, hence $x'=x^-$.

\bigskip
\textbf{Part (ii): }Again by condition {\bf{(D1)}} it suffices to show
$\lim_{s\to\infty}N_s=S$ in probability. 
Using Markov's inequality we can estimate
$\P[|N_s-S|>\e]\le\e^{-1}\E[|N_s-S|]$,
and it suffices to prove that the r.h.s.~tends to zero as $s$ tends to infinity. We start by estimating
$\E[|N_s-S|]\le\E[|N_s-\E N_s|]+|\E N_s-S|$, where the second summand tends to zero by part (i). Further we can write
\begin{align}\label{1Part2.2}
\E[|N_s-\E N_s|]=\E[\one_{E_{s,\e}^{\mathsf{D}}}|N_s-\E N_s|]+\E[\one_{(E_{s,\e}^{\mathsf{D}})^c}|N_s-\E N_s|],
\end{align}
where the second summand is in $o(1)$ by Lemma \ref{Thm2Lemma2}.
Using Lemma \ref{Thm2Lemma1}, and distinguishing between the cases $N_s\ge\E N_s$ and $N_s\le\E N_s$, the first summand in \eqref{1Part2.2} can be bounded above by
\begin{align}\label{1MainEstPart2.2}
\E[|N^+_s-\E N_s|]+\E[|N^-_s-\E N_s|],
\end{align}
where we write for convenience 
$$N^+_s=i_s^{-1}\sum_{X_j\in X^{(s)}\cap\, R^{\ms{left},+}_s(x)}\mu^{(s)}(X_j)\quad\text{ and }\quad N^-_s=i_s^{-1}\sum_{X_j\in X^{(s)}\cap\, R^{\ms{left},-}_{s,\e}(x)}\mu^{(s)}(X_j),$$
and note that the dependence on $\e$ in $N^-_s$ is omitted in the notation.  Then, the first summand in \eqref{1MainEstPart2.2} can be bounded by $\E[|N^+_s-\E N^+_s|]+|\E N^+_s -\E N_s|$,
and similarly the second summand in \eqref{1MainEstPart2.2} can be bounded by 
$\E[|N^-_s-\E N^-_s|]+|\E N^-_s-\E N_s|.$  In particular, the second summands in those expressions tend to zero when first $s$ is sent to infinity and then $\e$ is sent to zero as proved in part (i). In order to prove that the first summands in those expressions tend to zero as well, let us consider the two cases given in the theorem separately. 

\medskip
\textbf{Part (ii) case 1: }Let us start with the case of a possibly inhomogeneous Poisson point process $X^{(s)}$. Let $\mathbb{V}$ denote the variance with respect to the Poisson point process $X^{(s)}$. Then, by Jensen's inequality and Campbell's theorem for the variance we have 
\begin{align*}
	\E[|N^+_s-\E N^+_s|]&\le\big(\mathbb{V}[N^+_s]\big)^{1/2}=i_s^{-1}\big(\mathbb{V}[\sum_{X_j\in X^{(s)}\cap\, R^{\ms{left},+}_s(x)}\mu^{(s)}(X_j)]\big)^{1/2}\\
&=i_s^{-1}\big(s^d\int_{R^+_s(x)/s}\one_{\pi_1(y)\le\pi_1(x)}\mu(y)^2\la(y)\d y\big)^{1/2},
\end{align*}
where the expression under the square root is proportional to $i_s$ and hence the last line tends to zero as $s$ tends to infinity.
The same argument also holds for $N^+$ replaced by $N^-$.

\medskip
\textbf{Part (ii) case 2: }For the case of ergodic point processes we use the ergodic theorem, see for example~\cite[Theorem 6.2]{stoyanstochastic}. Let us abbreviate $r_s^+(x)=\nu_{d}(R^{\ms{left},+}_s(x))$. Then, by translation invariance,
\begin{align*}
	\E[|N^+_s-\E N^+_s|]=\mu  r_s^+(x) i_s^{-1}\E[|\frac{\#(X\cap\, R^{\ms{left},+}_s(x))}{r_s^+(x)}-\la|],
\end{align*}
where $r_s^+(x)/i_s$ tends to $\int^{0}_{-\infty}\one_{D}(x+r e_1)\d r$ as $s$ tends to infinity. 
In order to apply the ergodic theorem, let us first note that since $X$ is translation invariant, $R^{\ms{left},+}_s(x)$ can be replaced by 
$$R^{\ms{left},+}_{s,x}=\overline{\{sy\in s (D-x)\cap R^+_s(o):\, \pi_1(y)\le 0\}},$$ 
which is a convex and compact set that also has the property of eventually containing arbitrarily large balls.
However, due to boundary effects, it is not clear that $R^{\ms{left},+}_{s,x}$ contains an increasing sequence of sets for some $s$ tending to infinity. In order to appropriately modify $R^{\ms{left},+}_{s,x}$, let us define, similarly to $\zeta^-_s(\e)$, the quantity
$$\zeta^-_{s,x}=\sup\{\pi_1(z):\, {z\in\partial(D-x)\cap\, (R^+_s(o)/s)\text{ and } \pi_1(z)\le 0}\}$$  
to be the largest first coordinate of points in the part of the boundary of $s(D-x)$ intersected with the cylinder  $R^+_s(o)$ which lies to the left of the origin. Then, 
$$\hat R^{\ms{left},+}_{s,x}=\overline{\{sy\in {s}(D-x)\cap R^+_{s}(o):\, \zeta^-_{s,x}\le\pi_1(y)\le0\}}$$ 
is an increasing sequence of sets
due to the convexity of $D$. In particular $\hat R^{\ms{left},+}_{s,x}$ is a sequence of convex averaging windows as defined in~\cite{stoyanstochastic}. Denoting $\hat r_s^+(x)=\nu_{d}(\hat R^{\ms{left},+}_{s,x})$, we can estimate 
\begin{align*}
	\E|\frac{\#(X\cap\, R^{\ms{left},+}_{s,x})}{r_s^+(x)}-\la|														     \le 2\la\frac{r_s^+(x)-\hat r_s^+(x)}{\hat r_s^+(x)}+\E|\frac{\#(X\cap\, \hat R^{\ms{left},+}_{s,x})}{\hat r_s^+(x)}-\la|,
\end{align*}
where the second summand tends to zero by the ergodic theorem and the first summand tends to zero by convexity of $D$, using arguments similar to the ones used in the final paragraph of the proof of part (i). An analogous proof holds for the case where $N^+$ is replaced by $N^-$.
\end{proof}

\subsection{Proof of Theorem~\ref{trafficThmRad}}
Recall $\hat{v}=v/|v|$. For convenience we use the similar abbreviations as in the directed case and write $i_s(x)=s\nu_{d-1}(I^{\mathsf{R}}_{s}(x))$, $N_s=(i_s(x))^{-1}\sum_{X_i\in X^{(s)}\cap\Xi^{\mathsf{R}}_{s}(x)}\Delta(X_i)$ and $S= |x|^{-d+1}\int_{|x|}^{\infty}\lambda(r\hat{x})\mu(r\hat{x})r^{d-1}\d r$.
Let us define the following analogs of the cylinders $R^{\ms{left},+}_s(x)$ and $R^{\ms{left},-}_s(x)$ used in the directed case. First,
$$C^{+}_s(x)=\{sy\in sD\setminus B^d_{|sx|}(o):\, ||x|\hat y-x|\le (g(s)+2h(s))/s\}$$
is the set of points which lie in an extended cone around $sx$ facing towards the boundary of $D$. Second, for $s$ such that $g(s)\ge2h(s)$
$$C^{-}_s(x)=\{sy\in sD\setminus B^d_{|sx|}(o):\, ||x|\hat y-x|\le (g(s)-2h(s))/s\}$$
is the set of points which lie in an diminished cone around $sx$ facing towards the boundary of $D$, see also Figure~\ref{ConeFig}.

\begin{figure}[!htpb]
\centering
\begin{tikzpicture}[scale=1.0]

\fill[black!20!white] 
  (-5.8,6) -- (0,0) -- (5.8,6);

\fill[black!50!white]
  (-1.57,6) -- (0,0) -- (1.57,6);

\begin{scope}
\clip (-5,4) rectangle (5,0);
\draw (0,4) circle (1);
\draw (0,4) circle (2);
\draw (0,4) circle (3);
\end{scope}

\begin{scope}
\clip (-5,5) rectangle (5,0);
\fill[white] (0,0) circle (4);
\end{scope}

\draw 
  (-5.8,6) -- (0,0) -- (5.8,6);
\draw 
  (-1.57,6) -- (0,0) -- (1.57,6);

\begin{scope}
\clip (-5,5) rectangle (5,0);
\draw (0,0) circle (4);
\end{scope}

\begin{scope}
\clip (-5,5) rectangle (5,3.5);
\draw[very thick] (0,0) circle (4);
\end{scope}

\draw (-3,4) -- (3, 4);

\coordinate[label=-135: {${sx}$}] (sx) at (0,4);
\coordinate[label=180: {${o}$}] (o) at (0,0);
\coordinate[label={270: {${I^{\mathsf{R}}_{s}(x)}$}},yshift=-0.0cm] (Xi) at (0,3.4);
\coordinate[label={90: {${g(s)}$}},yshift=-0.0cm] (h) at (0.8,4.15);
\coordinate[label={90: {${2h(s)}$}},yshift=-0.0cm] (g) at (2.5,4.15);
\draw[decorate,decoration={brace,mirror}] (-1.9,3.45)--(1.9,3.45);
\draw[decorate,decoration={brace}] (0,4.1)--(2,4.1);
\draw[decorate,decoration={brace}] (2,4.1)--(3,4.1);
\foreach \i in {sx,o}
{
\fill (\i) circle (1pt);
}

\end{tikzpicture}
\caption{Construction of the cones $C^{-}_{s}(x)$ (dark gray) and $C^{+}_s(x)$ (union of light and dark gray).}
\label{ConeFig}
\end{figure}
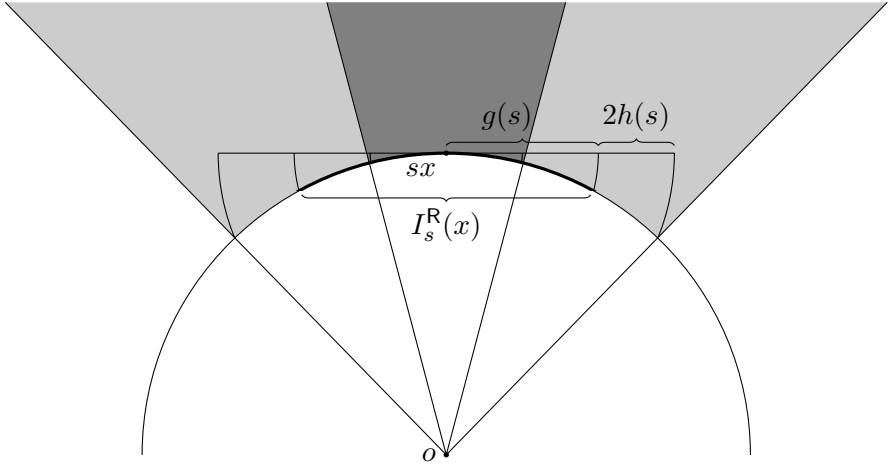

The next two lemmas give estimates of the accumulated traffic flow under $E_{s}^{\mathsf{R}}$ and $(E_{s}^{\mathsf{R}})^c$.

\begin{lemma}\label{Thm4Lemma1}
Let $x\in D$ and $X^{(s)}\in E_{s}^{\mathsf{R}}$, then
\begin{align*}
	\sum_{X_j\in X^{(s)}\cap\,  C^{-}_{s}(x)}\mu^{(s)}(X_j)\le \sum_{X_i\in X^{(s)}\cap\, \Xi^{\mathsf{R}}_{s}(x)}\Delta(X_i)\le\sum_{X_j\in X^{(s)}\cap\,  C^{+}_s(x)}\mu^{(s)}(X_j).
\end{align*}
\end{lemma}
\begin{proof}
For the upper bound, the cylinder condition given by $E_{s}^{\mathsf{R}}$ ensures that we can estimate
\begin{align*}
\sum_{X_i\in X^{(s)}\cap\,\Xi^{\mathsf{R}}_{s}(x)}\Delta(X_i)
&=\sum_{X_j\in X^{(s)}\cap\, C^{+}_s(x)}\mu^{(s)}(X_j)\sum_{X_i\in\,  \Xi^{\mathsf{R}}_{s}(x)}\one_{\g(X_j)}(X_i)\le\sum_{X_j\in X^{(s)}\cap\, C^{+}_s(x)}\mu^{(s)}(X_j)
\end{align*}
where for the last inequality we note that in any trajectory there is either one or no node that crosses $\Xi^{\mathsf{R}}_{s}(x)$ in the consecutive step. 
In order to justify the first equal sign, let us sketch the elementary proof. First note that it suffices to show that if $Z^{\ms{R}}_{h(s)}(y)\cap\, I^{\mathsf{R}}_{s}(x)\ne\es$, then $y\in C^{+}_s(x)$. To show this, let $z\in Z^{\ms{R}}_{h(s)}(y)\cap\, I^{\mathsf{R}}_{s}(x)$ and note that, since 
$\langle\hat z,\hat y\rangle$ is close to one, by the polarization identity $|\hat z-\hat y|\le 2|\hat z-\langle\hat z,\hat y\rangle \hat y|$. Therefore, 
$$|\hat y-\hat x|\le|\hat z-\hat y| +|\hat z-\hat x|\le2|\hat z-\langle\hat z,\hat y\rangle \hat y| +|\hat z-\hat x| .$$
Since $z\in Z^{\ms{R}}_{h(s)}(y)\cap\, I^{\mathsf{R}}_{s}(x)$ we have $|z|=s|x|$, $s|x||\hat z-\hat x|\le g(s)$ and $s|x||\hat z-\langle\hat z,\hat y\rangle \hat y|\le h(s)$ and thus $s|x||\hat y-\hat x|\le g(s)+2h(s)$. For the lower bound, similar arguments apply. 
\end{proof}

The following lemma is proved precisely as Lemma \ref{Thm2Lemma2}.

\begin{lemma}\label{Thm4Lemma2} It holds that $\E[\one_{(E_{s,\e}^{\mathsf{R}})^c}\sum_{X_i\in \Xi^{\mathsf{R}}_{s}(x)}\Delta(X_i)]\in O(1)$.
\end{lemma}

\begin{proof}[Proof of Theorem \ref{trafficThmRad}]
\textbf{Part (i): }By condition {\bf{(R1)}} it suffices to show $\lim_{s\to\infty}\E N_s=S$.
Using Lemma \ref{Thm4Lemma1}, Lemma \ref{Thm4Lemma2} and Campbell's theorem we can estimate for the upper bound
\begin{align*}
\E &N_s=\E\one_{E_{s}^{\mathsf{R}}} N_s+\E\one_{(E_{s}^{\mathsf{R}})^c}N_s\cr
	&\le i_s(x)^{-1}\E\sum_{X_j\in X^{(s)}\cap\, C^{+}_s(x)}\mu^{(s)}(X_j)+o(1)=i_s(x)^{-1}\int_{C^{+}_s(x)}\mu^{(s)}(y)\la^{(s)}(y)\d y+o(1).
\end{align*}
For the integral we can further use coordinate transformations and the co-area formula~\cite{coarea} to rewrite
\begin{align*}
	\int_{C^{+}_s(x)}\mu^{(s)}(y)\la^{(s)}(y)\d y&=\int\one_{|y|>|sx|}\one_{|sx||\hat y-\hat x|\le g(s)+2h(s)}\mu^{(s)}(y)\la^{(s)}(y)\d y\cr
&=s^{d}\int\one_{|y|>|x|}\one_{|\hat y-\hat x|\le (g(s)+2h(s))/|sx|}\mu(y)\la(y)\d y\cr
	&=s^{d}\int_{|x|}^\infty\int_{\partial B_r^d(o)}\one_{|\hat y-\hat x|\le (g(s)+2h(s))/|sx|}\mu(y)\la(y)\nu_{d-1}(\d y)\d r\cr
&=s^d\int_{|x|}^\infty r^{d-1}\int_{\partial B_1^d(o)}\one_{|\hat y-\hat x|\le (g(s)+2h(s))/|sx|}\mu(r y)\la(r y)\nu_{d-1}(\d y)\d r.
\end{align*}
Next, we identify $i_s(x)$ as the correct scaling for the inner integral above. More precisely,
\begin{align*}
i_s(x)=s\nu_{d-1}(I^{\mathsf{R}}_{s}(x))=s^d |x|^{d-1}\int_{\partial B_1^d(o)}\one_{|\hat y-\hat x|\le g(s)/|sx|}\nu_{d-1}(\d y).
\end{align*}
Since $h(s)\in o(g(s))$ and $\mu,\la$ are assumed to be continuous,
\begin{align*}
\frac{\int_{\partial B_1^d(o)}\one_{|\hat y-\hat x|\le (g(s)+2h(s))/|sx|}\mu(r y)\la(r y)\nu_{d-1}(\d y)}{\int_{\partial B_1^d(o)}\one_{|\hat y-\hat x|\le g(s)/|sx|}\nu_{d-1}(\d y)}
\end{align*}
is uniformly bounded in $s$ and $r$ and converges to $\mu(r \hat x)\la(r \hat x)$. Hence we can conclude by dominated convergence. 

\medskip
As for the lower bound, using the same arguments as above, we can estimate 
\begin{align*}
\E &N_s\ge s^di_s(x)^{-1}\int_{|x|}^\infty r^{d-1}\int_{\partial B_1^d(o)}\one_{|\hat y-\hat x|\le (g(s)-2h(s))/|sx|}\mu(r y)\la(r y)\nu_{d-1}(\d y)\d r-o(1).
\end{align*}
Since $h(s)\in o(g(s))$, the desired convergence result follows.

\bigskip
\textbf{Part (ii): }Again by condition {\bf{(R1)}} it suffices to show $\lim_{s\to\infty}N_s=S$
in probability and we can apply Markov's inequality as in the directed navigation case. Following the exact same arguments as in Theorem \ref{trafficThmDir}, replacing the following definitions 
$N^+_s=i_s(x)^{-1}\sum_{X_j\in X^{(s)}\cap\, C^{+}_s(x)}\mu^{(s)}(X_j)$ and $N^-_s=i_s(x)^{-1}\sum_{X_j\in X^{(s)}\cap\, C^{-}_{s}(x)}\mu^{(s)}(X_j)$, it suffices to show that  
$\E[|N^+_s-\E N^+_s|]$ and $\E[|N^-_s-\E N^-_s|]$
tend to zero as $s$ tends to infinity for Poissonian and ergodic point processes. The case of a possibly inhomogeneous Poisson point process $X^{(s)}$ can be proved using Jensen's inequality and Campbell's theorem for the variance as in the directed case. 

\medskip
For the case of ergodic point processes we have to construct a sequence of convex averaging windows as required for the application of the ergodic theorem, see \cite[Theorem 6.2]{stoyanstochastic}. 
Let us write $r^+_s(x)=\nu_{d}(C^{+}_s(x))$. Then, by translation invariance,
\begin{align*}
	\E[|N^+_s-\E N^+_s|]=\mu\frac{r^+_s(x)}{i_s(x)}\E[|\frac{\#(X\cap\, C^{+}_s(x))}{r^+_s(x)}-\la|],
\end{align*}
where $r^+_s(x)/i_s(x)$ tends to $|x|^{-d+1}\int_{|x|}^{\infty}\one_D(r\hat{x})r^{d-1}\d r$ as $s$ tends to infinity as proved in part (i). Note that $C^{+}_s(x)$ has to be modified in order to become a sequence of convex averaging windows. Let us define 
$\zeta_{s,x}=\inf\{\langle z,\hat x\rangle:\, z\in\partial D\cap\, (C^+_s(x)/s)\}$
to be the smallest component in the direction $\hat x$ of points in the boundary of $sD$ intersected with the cone $C^+_s(x)$. Then, 
$$\hat C^+_{s,x}=\overline{\{y\in D:\, sy\in C^+_{s}(x) \text{ and } |x|\le\langle y,\hat x\rangle\le\zeta_{s,x}\}}-sx$$ 
indeed is a sequence of convex averaging windows, see also Figure~\ref{ConeFig2}.
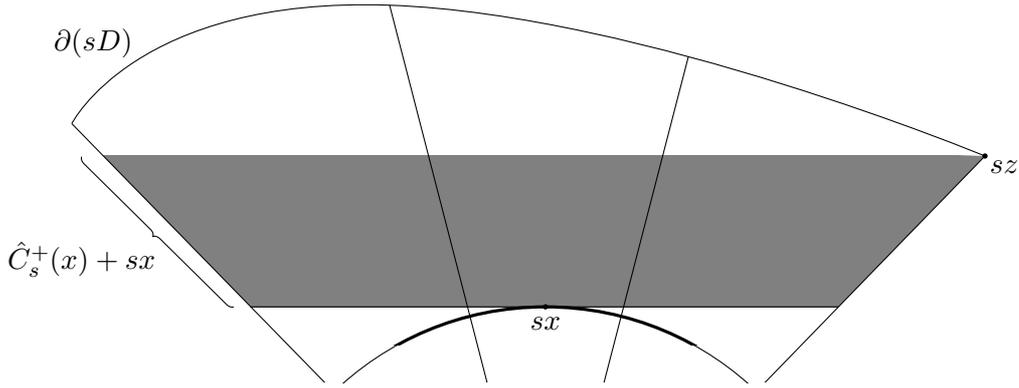
\begin{figure}[!htpb]
\centering
\begin{tikzpicture}[scale=1.0]

\begin{scope}
\clip (-8,3) rectangle (8,8.1);
\draw plot [smooth,thick,tension=1.0] coordinates {(-6.23,6.43) (-2.05,8) (5.78,6)};

\begin{scope}
\clip (-6,4) rectangle (5.78,6);
\fill[black!50!white]
  (-6.23,6.43) -- (0,0) -- (5.78,6);
\end{scope}

\begin{scope}
\clip (-6,4) rectangle (1.88,6);
\fill[black!50!white]
  (-2.05,8) -- (0,0) -- (1.88,7.32);
\end{scope}

\draw 
  (-6.23,6.43) -- (0,0) -- (5.78,6);
\draw 
  (-2.05,8) -- (0,0) -- (1.88,7.32);

\begin{scope}
\clip (-5,5) rectangle (5,0);
\draw (0,0) circle (4);
\end{scope}

\begin{scope}
\clip (-5,5) rectangle (5,3.5);
\draw[very thick] (0,0) circle (4);
\end{scope}

\draw (-3.88,4) -- (3.85, 4);

\coordinate[label=270: {$sx$}] (sx) at (0,4);
\coordinate[label={[label distance=-0.1cm]-45: {$sz$}},yshift=-0.0cm] (zse) at (5.78,6);
\coordinate[label={0: {${\partial (sD)}$}},yshift=-0.0cm] (sD) at (-6.6,7.5);
\coordinate[label={270: {${\hat C^{+}_{s}(x)+sx}$}},yshift=-0.0cm] (C+) at (-6.1,5);

\draw[decorate,decoration={brace}] (-4.1,4)--(-6.1,6);

\foreach \i in {sx,zse}
{
\fill (\i) circle (1pt);
}
\end{scope}
\end{tikzpicture}
\caption{Construction of the cone $\hat C^{+}_{s,x}$ (gray area), where $\langle z,\hat x\rangle=\zeta^x_s.$}
\label{ConeFig2}
\end{figure}

Denoting $\hat r^+_s(x)=\nu_{d}(\hat C^{+}_s(x))$ we can estimate 
\begin{align*}
	\E[|\frac{\#(X\cap\, C^{+}_s(x))}{r^+_s(x)}-\la|]&\le 2\la\frac{r^+_s(x)-\hat r^+_s(x)}{\hat r^+_s(x)}+\E[|\frac{\#(X\cap \hat C^{+}_s(x))}{\hat r^+_s(x)}-\la|]
\end{align*}
where the second summand tends to zero by the ergodic theorem. In the first summand there are two error terms contributing. The first term, measuring the error made close to the boundary of $D$ tends to zero by convexity of $D$. This can be seen using arguments similar to the ones used in the final paragraph of the proof of Theorem \ref{trafficThmDir} part (i). The other term, measuring the error made close to the origin, also tends to zero by the convexity of $B^d_{|sx|}(sx)$.

\medskip
Similar arguments apply for the case where $N^+$ is replaced by $N^-$.
\end{proof}

%% file: Prop.tex
\section{Proofs of Proposition \ref{LinkDensity} and Proposition \ref{InhomBalli}}

In the following, we write  $\mc{A}$ for the directed-spanning tree navigation and $B_r(x)$ (instead of the more verbose $B_r^d(x)$) for the $d$-dimensional ball with radius $r>0$ centered at $x\in\R^d$. To begin with, we prove Proposition~\ref{LinkDensity}. In a stationary setting, according to well-known results on intersection processes, the function $\lambda_{\mc{A}}$ exists and is constant on $D$. In fact, one can derive an explicit formula for $\lambda_{\mc{A}}$ that depends only on the length intensity and the directional distribution of the segment process of links, see~\cite[Theorem 4.5.3]{sWeil}. Now, the proof is based on the observation that directed spanning trees are \emph{strongly stabilizing} in the sense of~\cite{stab2,stab3}, so that locally around a given point $sx\in sD$ the process $X^{(s)}$ can be replaced by a homogeneous Poisson point process with intensity $\lambda(x)$. 

In order to make this precise, we first recall the standard coupling of Poisson point processes with bounded intensities. Let $X$ be a Poisson point process in $\R^d\times[0,\lambda_{\ms{max}}]$ whose intensity measure is given by $\nu_{d+1}(\cdot)$, where $\lambda_{\ms{max}}=\max_{x\in D}\lambda(x)$. If $\lambda^*:\R^d\to[0,\lambda_{\ms{max}}]$ is any measurable function, then we define 
$X^{[\lambda^*]}=\{X_i:\,(X_i,U_i)\in X\text{ and }U_i\le \lambda^*(X_i)\}$
noting that $X^{[\lambda^*]}$ is a Poisson point process in $\R^d$ with intensity function $\lambda^*$.  For instance, we can now express $X^{(s)}$ as $X^{[\one_{s D}(\cdot)\la^{(s)}]}$. If $\lambda^*$ is a function that is constant and equal to some $\lambda_0$, we also write $X^{[\la_0]}$ instead of $X^{[\la^*]}$.

To begin with, we state three auxiliary results (Lemmas~\ref{iPlusLem}--\ref{locHomLem}) and explain how Proposition~\ref{LinkDensity} can be derived using these results. Afterwards, we prove Lemmas~\ref{iPlusLem}--\ref{locHomLem}. First, we show that only points of $X^{(s)}$ which lie close to $I_s(x)=I^{\ms{D}}_s(x)$ are relevant. More precisely, fixing some $\xi'\in(\max\{\xi,1/2\},1)$ and putting $I_s^+(x)=I_s(x)\oplus B^d_{s^{1-\xi'}}(o)$, we have the following result.
\begin{lemma}
\label{iPlusLem}
Let $x\in D$ be arbitrary. Then,
\begin{enumerate}
	\item $\int_{sD\setminus I^+_s(x)}\P([y,\mc{A}(y,X^{(s)}\cup\{y\})]\cap I_s(x)\ne\es)\lambda^{(s)}(y)\d y\in o(\nu_{d-1}(I_s(x))),$
\item $\int_{\R^d\setminus I^+_s(x)}\P([y,\mc{A}(y,X^{[\lambda(x)]}\cup\{y\})]\cap I_s(x)\ne\es)\d y\in o(\nu_{d-1}(I_s(x))).$
\end{enumerate}
\end{lemma}

Second, we provide an elementary result showing that if we replace integrals over $I_s^+(x)$ with respect to the intensity $\lambda^{(s)}$ by integrals with respect to the constant intensity $\lambda(x)$, then the error is of order $o(\nu_{d-1}(I_s(x)))$.
\begin{lemma}
\label{intChangeLem}
Let $x\in D$ be arbitrary and let $f:I^+_s(x)\to[0,1]$ be a measurable function. Then,
$$\Big|\int_{I^+_s(x)}f(y)\lambda^{(s)}(y)\d y-\lambda(x)\int_{I^+_s(x)}f(y)\d y\Big|\in o(\nu_{d-1}(I_s(x))).$$
\end{lemma}

Third, we show that replacing the Poisson point process $X^{(s)}$ in the probability $\P([y,\mc{A}(y,X^{(s)}\cup\{y\})]\cap I_s(x)\ne\es)$ by the homogeneous Poisson point process $X^{[\lambda(x)]}$ leads to a negligible error.

\begin{lemma}
\label{locHomLem}
Let $x\in D$ be arbitrary and assume that $\la$ is locally Lipschitz. Then, 
$$\int_{I^+_s(x)}\hspace{-0.19cm}\big|\P([y,\mc{A}(y,X^{(s)}\cup\{y\})]\cap I_s(x)\ne\es)-\P([y,\mc{A}(y,X^{[\lambda(x)]}\cup\{y\})]\cap I_s(x)\ne\es)\big|\d y\hspace{-0.05cm}\in\hspace{-0.05cm} o\big(\nu_{d-1}(I_s(x))\big).$$ 
\end{lemma}

Using Lemmas~\ref{iPlusLem}--\ref{locHomLem}, we can now prove Proposition~\ref{LinkDensity}.

\begin{proof}[Proof of Proposition~\ref{LinkDensity}]
Let $x\in D$ be arbitrary. We claim that $\lambda_{\mc{A}}(x)$ equals the intensity of the intersection process of the stationary segment process 
$\{(X_i,[X_i,\mc{A}(X_i,X^{[\lambda(x)]})])\}_{X_i\in X^{[\lambda(x)]}}$
with the hyperplane $\{y\in\RR: \pi_1(y)= \pi_1(x)\}$. 

First, using the Slivnyak-Mecke theorem~\cite[Theorem 3.2.3]{sWeil}, this intensity is expressed as
$$\frac{\lambda(x)\int_{\RR}\P([y,\mc{A}(y,X^{[\lambda(x)]}\cup\{y\})]\cap I_s(x)\ne\es)\d y}{\nu_{d-1}(I_s(x))},$$
which is independent of $s$ by stationarity.
Moreover, again by the Slivnyak-Mecke theorem, 
\begin{align*}
\E\#\{X_i\in X^{(s)}:[X_i,\mc{A}(X_i,X^{(s)})]\cap I_s(x)\ne\es\}&\hspace{-0.1cm}=\hspace{-0.1cm}\int_{sD}\hspace{-0.2cm}\P([y,\mc{A}(y,X^{(s)}\cup\{y\})]\cap I_s(x)\ne\es)\lambda^{(s)}(y)\d y.
\end{align*}

Therefore, by Lemma~\ref{iPlusLem}, it suffices to show that the difference
$$\int_{I^+_s(x)}\hspace{-0.4cm}\P([y,\mc{A}(y,X^{(s)}\cup\{y\})]\cap I_s(x)\ne\es)\lambda^{(s)}(y)\d y-\int_{I^+_s(x)}\hspace{-0.4cm}\P([y,\mc{A}(y,X^{[\lambda(x)]}\cup\{y\})]\cap I_s(x)\ne\es)\lambda(x)\d y.$$
is of order $o(\nu_{d-1}(I_s(x)))$. 
Now, by Lemma~\ref{intChangeLem} this assertion is reduced to the statement  that 
$$\int_{I^+_s(x)}\Big|\P([y,\mc{A}(y,X^{(s)}\cup\{y\})]\cap I_s(x)\ne\es)-\P([y,\mc{A}(y,X^{[\lambda(x)]}\cup\{y\})]\cap I_s(x)\ne\es)\Big|\d y.$$
is of order $o(\nu_{d-1}(I_s(x)))$. Hence, an application of Lemma~\ref{locHomLem} completes the proof.
\end{proof}

Finally, we provide the proofs for Lemmas~\ref{iPlusLem}--\ref{locHomLem}.

\begin{proof}[Proof of Lemma~\ref{iPlusLem}]
We only prove the first assertion, since the second may be deduced using similar arguments. Let $\e>0$ be such that $B_{3\e}(x)\subset D$. First, if $y\in sD\setminus B_{2\e s}(sx)$ is such that $[y,\mc{A}(y,X^{(s)}\cup\{y\})]\cap I_s(x)\ne\es$, then $X^{(s)}$ does not contain any points in the set 
$$B_{|y-sx|-\e s}(y)\cap B_{2\e s}(sx)\cap([\pi_1(y),\infty)\times\R^{d-1}).$$
Now, an elementary argument shows that this set contains a half-ball of radius $2^{-1}\e s$. Hence,
%
\begin{align}
\label{expDecY}
\P([y,\mc{A}(y,X^{(s)}\cup\{y\})]\cap I_s(x)\ne\es)\le \exp\big(-\la_{\ms{min}}\kappa_d2^{-1}(2^{-1}\e s)^d\big),
\end{align}
where $\lambda_{\ms{min}}=\min_{x\in D}\lambda(x)$ and $\kappa_d$ denotes the volume of the unit ball in $\R^d$. Hence, it suffices to show that 
$$\int_{B_{2\e s}(sx)\setminus I^+_s(x)}\P([y,\mc{A}(y,X^{(s)}\cup\{y\})]\cap I_s(x)\ne\es)\lambda^{(s)}(y)\d y\in o(\nu_{d-1}(I_s(x))).$$
In order to prove this assertion, we observe that
\begin{align*}
&\int_{B_{2\e s}(sx)\setminus I^+_s}\P([y,\mc{A}(y,X^{(s)}\cup\{y\})]\cap I_s(x)\ne\es)\lambda^{(s)}(y)\d y\\
&\quad\le \int_{B_{2\e s}(sx)}\P(|y-\mc{A}(y,X^{(s)}\cup\{y\})|\ge s^{1-\xi'})\lambda^{(s)}(y)\d y.
\end{align*}
Proceeding as in~\eqref{expDecY}, the probability in the integrand is at most
$\exp\big(-\la_{\ms{min}}\kappa_d2^{-1}s^{d-d\xi'}\big),$
so that the proof is concluded by noting that 
\begin{align*}
	\int_{B_{2\e s}(sx)}\P(|y-\mc{A}(y,X^{(s)}\cup\{y\})|\ge s^{1-\xi'})\lambda^{(s)}(y)\d y\in o(\nu_{d-1}(I_s(x))).&\qedhere
\end{align*}
\end{proof}

\begin{proof}[Proof of Lemma~\ref{intChangeLem}]
First, letting $L$ denote the Lipschitz constant of $\lambda$ in $I^+_s(x)$, we see that replacing $\lambda^{(s)}(y)$ by $\lambda(x)$ leads to the error term 
\begin{align*}
\int_{I^+_s(x)}|\lambda^{(s)}(y)-\lambda(x)|\d y\le L(s^{\xi}+s^{1-\xi'})s^{-1}\nu_{d}(I^+_s(x)).
\end{align*}
Since $\nu_{d}(I^+_s(x))$ is of order $\nu_{d-1}(I_s(x))s^{1-\xi'}$, we conclude that after division by $\nu_{d-1}(I_s(x))$ the last line tends to zero as $s$ tends to infinity. 
\end{proof}

\begin{proof}[Proof of Lemma~\ref{locHomLem}]
First, we put $I^{++}_s(x)=I^+_s(x)\oplus B_{s^{1-\xi'}}(o)$
and $\alpha=(\xi'-\xi)/(2d)$.
Hence, for every $y\in I^+_s(x)$,
\begin{align*}
&|\P([y,\mc{A}(y,X^{(s)}\cup\{y\})]\cap I_s(x)\ne\es\})-\P([y,\mc{A}(y,X^{[\lambda(x)]}\cup\{y\})]\cap I_s(x)\ne\es\})|\\
&\quad\le \P(\mc{A}(y,X^{(s)}\cup\{y\})\ne \mc{A}(y,X^{[\lambda(x)]}\cup\{y\}))\\
&\quad\le\P(X^{[\lambda(x)]}\cap B_{s^\alpha}(y)\ne X^{(s)}\cap B_{s^\alpha}(y))+\P(|y-\mc{A}(y,X^{[\lambda(x)]}\cup\{y\})|\ge s^{\alpha}).
\end{align*}
Now, as in the proof of Lemma~\ref{iPlusLem}, we see that
$$\int_{I^+_s(x)} \P(|y-\mc{A}(y,X^{[\lambda(x)]}\cup\{y\})|\ge s^{\alpha})\d y= \P(|\mc{A}(o,X^{[\lambda(x)]}\cup\{o\})|\ge s^{\alpha})\nu_d(I^+_s(x))$$
is of order $O(s^{-n})$ for any given $n\ge1$. Moreover, the symmetric difference $(X^{(s)}\D \, X^{[\la(x)]})\cap B_{s^\alpha}(y)$ is a Poisson point process with intensity function $y\mapsto |\lambda^{(s)}(y)-\lambda(x)|$. Now, the Lipschitz continuity implies that 
$\sup_{y\in I^{++}_s(x)}|\lambda^{(s)}(y)-\lambda(x)|\le L(s^{\xi}+2s^{1-\xi'})s^{-1}.$
Hence, 
\begin{align*}
\P(X^{(s)}\cap B_{s^\alpha}(y)\ne X^{[\la(x)]}\cap B_{s^\alpha}(y))&\le \E\#\big((X^{(s)}\D \, X^{[\la(x)]})\cap B_{s^\alpha}(y)\big)\\
&\le L(s^{\xi}+2s^{1-\xi'})s^{-1}\nu_{d}(B_{s^{\alpha}}(y)).
\end{align*}
We conclude the proof by noting that by the definition of $\alpha$, the right-hand side is of order 
\begin{align*}
	o(s^{\xi'-1})=o\big(\nu_d(I^+_s(x))^{-1}\nu_{d-1}(I_s(x))\big).&\qedhere
\end{align*}
\end{proof}

The main idea for proving Proposition~\ref{InhomBalli}, is to first show, using the Lipschitz property, that locally $X^{(s)}$ looks like a homogeneous Poisson point process and then to apply known results from the homogeneous setting considered in~\cite{BB:2007}. Before we start with the proof, we present an auxiliary result showing that with high probability long edges in  the directed spanning tree on $X^{(s)}$ can only appear at far right points of $sD$. 
More precisely, as observed in~\cite{BB:2007}, the directed spanning tree clearly enjoys a strong stabilization property. In the current setting this means the following. If $\varphi\subset\R^d$ is any locally finite set such that $X_i,\mc{A}(X_i,X^{(s)})\in\varphi$ and there does not exist $x\in\varphi$ with $\pi_1(x)>\pi_1(X_i)$ and $|X_i-x|<|\mc{A}(X_i,X^{(s)})-X_i|$, then $\mc{A}(X_i,X^{(s)})=\mc{A}(X_i,\varphi)$. Moreover, the following result shows that the maximal radius of stabilization 
$$R_{-,s}=\max_{X_i\in X^{(s)}:\,  \pi_1(X_i)\le s/2-s^{1/(2d)}}|X_i-\mc{A}(X_i,X^{(s)})|$$
 is small with high probability. 
More precisely, we let $E^{(1)}_s$ denote the event $\{R_{-,s}\ge s^{1/(4d)}\}$.

\begin{lemma}
\label{stabiLem}
As $s\to\infty$, $\P(E^{(1)}_s)\in O(s^{-2d})$.
\end{lemma}
\begin{proof}
For any $x\in [-s/2,s/2]^d$ with $\pi_1(x)\le s/2-s^{1/(2d)}$ we have that 
\begin{align*}
\P(|x-\mc{A}(x,X^{(s)}\cup\{x\})|>s^{1/(4d)})&=\P(X^{(s)}\cap B_{s^{1/(4d)}}(x)\cap([\pi_1(x),\infty)\times\R^{d-1})=\es)\\
&\le \exp(-\kappa_d2^{-d}s^{1/4}\min_{x\in D}\lambda(x)).
\end{align*}
Hence, the claim follows from the Slivnyak-Mecke formula.
\end{proof}
In the following,  we write 
$$D_{-,s}=[-\tfrac s2,\tfrac s2-2s^{1/(2d)}]\times[-\tfrac s2+2s^{1/(2d)},\tfrac s2-2s^{1/(2d)}]^{d-1}$$
for the domain $sD$ shrunk except for the left boundary. Furthermore, to simplify notation, we replace $s^{1/(2d)}$ by $s'$ and assume that $s(s')^{-1}$ is an odd integer. This is not a restriction, since otherwise $s'$ can be adjusted in such a way that it is of the same order as $s^{1/(2d)}$.
Next, we quantify the local homogeneity of the Poisson point process $X^{(s)}$ by comparing it to a homogeneous version using coupling. More precisely, for $s\ge1$ and $z\in\Z^d$ we let $H_{s,z}$ denote the event that the point processes $X^{(s)}$ and $X^{[\lambda_{s,z}]}$ agree on $Q_{3s'}(s'z)\cap sD$, where 
$$\lambda_{s,z}=\max\big\{\lambda^{(s)}(x):\,x\in Q_{3s'}(s'z)\big\}$$ 
denotes the maximum of the intensity $\lambda^{(s)}(\cdot)$ in the cube $Q_{3s'}(s'z)$ of side-length $3s'$ centered at $s'z$. In other words, under the event $H_{s,z}$ the Poisson point process $X^{(s)}$ cannot be distinguished from a homogeneous Poisson point process with intensity $\lambda_{s,z}$ in a $3s'$-environment around $s'z$. Finally, we say that the site $z$ (or the associated cube $Q_{s'}(s'z)$) is \emph{$s$-good},  if the event $H_{s,z}$ occurs. Sites that are not $s$-good are called \emph{$s$-bad}.

We use that locally, trajectories do not deviate substantially from the horizontal line. More precisely, let $H'_{s,z}$ denote the event that for every path $\Gamma$ in the directed spanning tree on $X^{[\lambda_{s,z}]}\cap Q_{3s'}(s'z)$ whose starting point $X_0$ is contained in $Q_{s'}(s'z)$ and whose endpoint $X_{\ms{end}}$ satisfies $\pi_1(X_{\ms{end}})\le\pi_1(s'z)+s'$ we have 
$\Gamma\subset Z^{\ms{D}}_{(s')^{5/8}}(X_{0}).$
Then, we say that the event $E^{(2)}_s$ occurs if there exists $z\in\Z^d$ such that $s'z\in sD_{-,s}$ and $H'_{s,z}$ fails to occur. In particular,~\cite[Theorem 4.10]{BB:2007} gives the following auxiliary result.
\begin{lemma}
\label{homFluctLem}
As $s\to\infty$, $\P(E^{(2)}_s)\in O(s^{-2d})$. 
\end{lemma}

Now, for any path $\Gamma$ in the directed spanning tree on $X^{(s)}$ we let 
$$\#_s\Gamma=\#\{z\in\Z^d:\, X_i\in Q_{s'}(s'z)\text{ for some }X_i\in\Gamma\}$$
denote the number of $s'$-cubes intersected by $\Gamma$. Similarly, we let $\#_{s,\ms{g}}\Gamma$ and $\#_{s,\ms{b}}\Gamma$ denote the number of good, respectively bad cubes that are intersected by $\Gamma$. Next, we provide an upper bound for the vertical displacement of a path $\Gamma$ in terms of $\#_{s,\ms{g}}\Gamma$ and $\#_{s,\ms{b}}\Gamma$. 
To make this precise, if $\Gamma$ is a path in the directed spanning tree on $X^{(s)}$ starting from $X_0\in X^{(s)}$, then we let 
$V(\Gamma)=\max_{X_i\in\Gamma}d_{1,\infty}(X_0,X_i)$
denote the maximal vertical displacement of $\Gamma$, where we write 
$d_{1,\infty}(X_0,X_i)= d_\infty(X_0+\R e_1,X_i),$
for the $d_\infty$-distance of $X_i$ to the horizontal line $X_0+\R e_1$. 

\begin{lemma}
\label{mDetBound}
Let $\Gamma\subset D_{-,s}$ be an arbitrary path in the directed spanning tree on $X^{(s)}$. Then, almost surely under the complement of the event $E^{(1)}_s\cup E^{(2)}_s$,
$$V(\Gamma)\le2s'+3(s')^{5/8}\#_{s,\ms{g}}\Gamma +3s'\#_{s,\ms{b}}\Gamma.$$
\end{lemma}
\begin{proof}
By shortening the path if necessary, we may assume that the maximal vertical displacement $V(\Gamma)$ is achieved at the endpoint $X_{\ms{end}}$ of $\Gamma$. The proof proceeds via induction on the number of vertices in $\Gamma$. The assertion is trivial if $V(\Gamma)\le2s'$, so that we may assume that $V(\Gamma)>2s'$. Fix the site $z_0\in\Z^d$ such that $Q_{s'}(s'z_0)$ contains the starting point $X_0$ of $\Gamma$. 
Then, we let $X_2\in\Gamma$ denote the first vertex of $\Gamma$ such that $X_2\in Q_{s'}(s'z_2)$ for some $z_2\in\Z^d$ with $d_{1,\infty}(z_0,z_2)>1$. We also let $X_1\in\Gamma$ be such that $X_2=\mc{A}(X_1,X^{(s)})$. That is, $X_1$ is the predecessor of $X_2$ in $\Gamma$. Similarly, we let $X_3\in\Gamma$ denote the last vertex of $\Gamma$ such that $X_3\in Q_{s'}(s'z_3)$ for some $z_3\in\Z^d$ with $d_{1,\infty}(z_0,z_3)\le1$. Finally, we put $X_4=\mc{A}(X_3, X^{(s)})$. We refer to Figure~\ref{mDetFig} for an illustration of the construction.

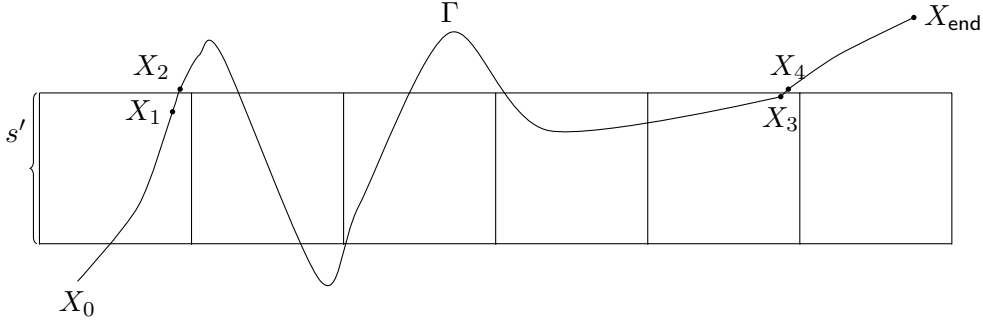
\begin{figure}[!htpb]
\centering
\begin{tikzpicture}[scale=1.0]
\draw[step=2cm] (0,-0.01) grid (12,2);
\coordinate[label=-90: $X_0$] (x0) at (0.5,-0.5);
\coordinate[label=180: $X_1$] (x1) at (1.75,1.75);
\coordinate[label=135: $X_2$] (x2) at (1.85,2.05);
\coordinate[label=-90: $X_3$] (x3) at (9.75,1.95);
\coordinate[label=90: $X_4$] (x4) at (9.85,2.05);
\coordinate[label=0: $X_{\ms{end}}$] (xe) at (11.5,3);
\coordinate[label=90: $\Gamma$] (xp) at (5.4,2.8);
\fill[black] (x1) circle (1pt);
\fill[black] (x2) circle (1pt);
\fill[black] (x3) circle (1pt);
\fill[black] (x4) circle (1pt);
\fill[black] (xe) circle (1pt);
\draw plot [smooth,thick,tension=0.5] coordinates {(x0) (1.3,0.5) (x1) (x2) (2.1,2.5) (2.4,2.5) (3.7,-0.5) (4.2,0.5) (xp) (6.7,1.5) (x3)};
\draw (x3)--(x4);
\draw plot [smooth,thick,tension=0.5] coordinates { (x4) (10.5,2.5) (xe)};
\draw[decorate,decoration={brace},xshift=-1pt] (0,-0.0)--(0,2.0);
\coordinate[label=180: $s'$,xshift=-1pt] (sp) at (0,1.5);

\end{tikzpicture}
\caption{Illustration of the induction step in the proof of Lemma~\ref{mDetBound}.}
\label{mDetFig}
\end{figure}

In particular, no $s'$-subcube is hit by both $\Gamma[X_0,X_1]$ and $\Gamma[X_4,X_{\ms{end}}]$, where $\Gamma[X_0,X_1]$ and $\Gamma[X_4,X_{\ms{end}}]$ denote the subpaths of $\Gamma$ from $X_0$ to $X_1$ and from $X_4$ to $X_{\ms{end}}$, respectively.

Now, by the definitions of the point $X_3$ and the event $E^{(1)}_s$,
$$V(\Gamma)\le V(\Gamma[X_4,X_{\ms{end}}])+2s'+|X_3-X_4|\le V(\Gamma[X_4,X_{\ms{end}}])+\tfrac{5}{2}s'.$$
Hence, using the induction hypothesis, we arrive at 
\begin{align*}
V(\Gamma)&\le 2s'+3(s')^{5/8}\#_{s,\ms{g}}\Gamma[X_4,X_{\ms{end}}]+3s'\#_{s,\ms{b}}\Gamma[X_4,X_{\ms{end}}]+\tfrac{5}{2}s'.
\end{align*}
In particular, the assertion follows if $\Gamma[X_0,X_1]$ hits at least one $s$-bad cube. Therefore, it remains to consider the case, where $\Gamma[X_0,X_1]$ intersects only $s$-good cubes. Since the complement of the event $E^{(1)}_s\cup E^{(2)}_s$ occurs, it follows that 
$$\tfrac{5}{6}s'\le d_{1,\infty}(X_0,X_1)
\le V(\Gamma[X_0,X_1])\le (s')^{5/8}\#_{s,\ms{g}}\Gamma[X_0,X_1].$$
Therefore, we complete the induction step by noting that
\begin{align*}
V(\Gamma)&\le 2s'+3(s')^{5/8}\#_{s,\ms{g}}\Gamma[X_4,X_{\ms{end}}]+3s'\#_{s,\ms{b}}\Gamma[X_4,X_{\ms{end}}]+\tfrac{5}{2}s'\\
&\le 2s'+3(s')^{5/8}(\#_{s,\ms{g}}\Gamma[X_0,X_1]+\#_{s,\ms{g}}\Gamma[X_4,X_{\ms{end}}])+3s'\#_{s,\ms{b}}\Gamma[X_4,X_{\ms{end}}]\\
&\le 2s'+3(s')^{5/8}\#_{s,\ms{g}}\Gamma +3s'\#_{s,\ms{b}}\Gamma,\qedhere
\end{align*}
\end{proof}

Hence, in order to prove Proposition~\ref{InhomBalli}, we need to derive appropriate upper bounds on $\#_{s,\ms{g}}\Gamma(X_i)$ and $\#_{s,\ms{b}}\Gamma(X_i)$ for any $X_i\in X^{(s)}$. 
First, we show that conditioned on the event that a path $\Gamma$ is not too long, the number of bad cubes $\#_{s,\ms{b}}\Gamma$ is of order $o(s/s')$ with high probability.
\begin{lemma}
\label{badBoundLem}
Almost surely, for every path $\Gamma\subset D_{-,s}$ in the directed spanning tree on $X^{(s)}$,
$$\one\{\#_s\Gamma\le 3^{d+3}s(s')^{-1}\}\P(\#_{s,\ms{b}}\Gamma\ge s^{1-5/(8d)}|X^{(s)})\le\exp(-\sqrt{s}).$$
\end{lemma}
A proof of Lemma~\ref{badBoundLem} will be given below. Second, we use a percolation argument to show that $\#_s\Gamma(X_i)$ is of order $O(s/s')$ with high probability.  More precisely, we let $\Gamma_{-,s}(X_i)$ denote the longest subpath of $\Gamma(X_i)$ that starts at $X_i$ and is contained entirely within $D_{-,s}$. Then, we let $E^{(3)}_s$ denote the event that there exists a point $X_i\in X^{(s)}$ such that $\#_s\Gamma_{-,s}(X_i)\ge 3^{d+3}s(s')^{-1}$.
\begin{lemma}
\label{linGrowthLem}
 As $s\to\infty$, $\P(E^{(3)}_s)\in O(s^{-2d})$.
\end{lemma}

Before proving Lemmas~\ref{badBoundLem} and~\ref{linGrowthLem}, we show how they can be used to deduce Proposition~\ref{InhomBalli}.

\begin{proof}[Proof of Proposition \ref{InhomBalli}]
The assertion involving $E^{\ms{D}}_{s,\e,1}$ is clear since the unique dead end can only occur in the right boundary of $D$. Indeed, the probability that $X^{(s)}$ does not contain points close to the right boundary of $sD$, decays exponentially in $s$. In order to deal with $E^{\ms{D}}_{s,\e,2}$, we first consider the auxiliary event 
$E^{\mathsf{D}_-}_{s,\e,2}=\{\Gamma_{-,s}(X_i)\subset Z^{\ms{D}}_{g(s)}(X_i)\text{ for all }X_i\in X^{(s)}\},$
where $g(s)=s^{1-1/(64d)}$ and assert that 
$E^{\mathsf{D}_-}_{s,\e,2}\subset E^{\mathsf{D}}_{s,\e,2}\cup\bigcup_{i=1}^3 E^{(i)}_s.$
In order to prove this assertion, we assume that the event 
$E^{\mathsf{D}_-}_{s,\e,2}\cap\Big(\bigcup_{i=1}^3 E^{(i)}_s\Big)^c$
occurs. It suffices to show under this event that whenever $X_i\in X^{(s)}$ is such that $X_i\in sD\setminus D_{-,s}$, then $\Gamma(X_i)\cap (sD)_{\e s}=\es$. Suppose the contrary and let $X_{i_2}\in X^{(s)}$ be the first point in $\Gamma(X_i)$ such that $X_{i_2}\in(sD)_{\e s}$. Moreover, let $X_{i_1}$ be the first point on $\Gamma(X_i)$ such that $\Gamma[X_{i_1},X_{i_2}]\subset (sD)_{\e s/2}$. Then, the vertical deviation of the subpath from $X_{i_1}$ to $X_{i_2}$ is larger than $\e s/4$, which contradicts the occurrence of $E^{\mathsf{D}_-}_{s,\e,2}$. This construction is illustrated in Figure~\ref{inHomFig}.

\begin{figure}[!htpb]
\centering
\begin{tikzpicture}[scale=1.0]

\draw (0,0)--(5,0);
\draw[dashed] (0,0.5)--(5,0.5);
\draw[dashed] (0,1.5)--(5,1.5);
\draw[dashed] (0,2.5)--(5,2.5);
\coordinate[label=180: $\partial (sD)$] (x1) at (0,0);
\coordinate[label=180: $\partial D_{-,s}$] (x2) at (0,0.5);
\coordinate[label=180: $\partial ((sD)_{\e s/2})$] (x3) at (0,1.5);
\coordinate[label=180: $\partial ((sD)_{\e s})$] (x4) at (0,2.5);

\coordinate[label=180: $X_{i}$,yshift=-0.1cm] (y1) at (1,0.4);
\coordinate[label=180: $\Gamma(X_{i})$] (y1a) at (1.4,1.2);
\coordinate[label=180: $X_{i_1}$] (y2) at (3,1.7);
\coordinate[label=180: $X_{i_2}$] (y3) at (5,2.7);
\draw plot [smooth,thick,tension=0.5] coordinates {(y1) (1.5,1.4) (2,0.3) (y2) (3.7,2.4) (4.5,1.6) (y3)};

\fill (y1) circle (1pt);
\fill (y2) circle (1pt);
\fill (y3) circle (1pt);

\end{tikzpicture}
\caption{Behavior of trajectories close to the boundary of $sD$}
\label{inHomFig}
\end{figure}
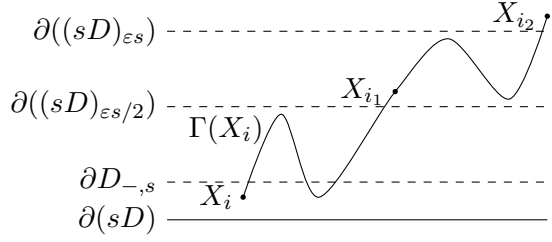

Hence, by Lemmas~\ref{stabiLem},~\ref{homFluctLem} and~\ref{linGrowthLem}, it remains to show that $1-\P(E^{\mathsf{D}_-}_{s,\e,2})\in O(s^{-2d})$.
First, we observe that if $X_i\in X^{(s)}\cap D_{-,s}$ is such that 
$\Gamma_{-,s}(X_i)\not\subset Z^{\ms{D}}_{s^{1-1/(64d)}}(X_i),$
then $V(\Gamma_{-,s}(X_i))\ge s^{1-1/(32d)}$. Therefore, Lemma~\ref{mDetBound} yields that 
\begin{align*}
1-\P(E^{\mathsf{D}_-}_{s,\e,2})&\le \P\Big(\sup_{X_i\in X^{(s)}\cap D_{-,s}}V(\Gamma_{-,s}(X_i))\ge s^{1-1/(32d)}\Big)\\
&\le\P(E^{(1)}_s\cup E^{(2)}_s\cup E^{(3)}_s)+\P\Big((E^{(3)}_s)^c\cap\Big\{\sup_{X_i\in X^{(s)}\cap D_{-,s}}\#_{s,\ms{b}}\Gamma_{-,s}(X_i)\ge s^{1-9/(16d)}\Big\}\Big).
\end{align*}
By Lemmas~\ref{stabiLem},~\ref{homFluctLem} and~\ref{linGrowthLem}, the first summand is of order $O(s^{-2d})$. 
For the second, we may condition on $X^{(s)}$ to deduce from Lemma~\ref{badBoundLem} that
\begin{align*}
&\P\big((E^{(3)}_s)^c\cap\big\{\#_{s,\ms{b}}\Gamma_{-,s}(X_i)\ge s^{1-9/(16d)}\text{ for some }X_i\in X^{(s)}\cap D_{-,s} \big\}\big)\\
&\quad\le\E\Big(\one\{(E^{(3)}_s)^c\}\sum_{X_i\in X^{(s)}\cap D_{-,s}}\P(\#_{s,\ms{b}}\Gamma_{-,s}(X_i)\ge s^{1-9/(16d)}|X^{(s)})\Big)\\
&\quad\le\exp(-\sqrt{s})\E\#X^{(s)}.
\end{align*}
Since for large $s$ the last expression is of order $O(s^{d}\exp(-\sqrt{s}))$, we conclude the proof.
\end{proof}
Now it remains to prove Lemmas~\ref{badBoundLem} and~\ref{linGrowthLem}. We begin with the first one. As an important auxiliary result, we show that conditioned on $X^{(s)}$, sites are $s$-good with high probability. Let $L$ denote the global Lipschitz constant of $\la$.

\begin{lemma}
\label{whpHpLem}
Almost surely, for every $z\in\Z^d$ with $s'z\in D_{-,s}$ it holds that 
$$1-\P(H_{s,z}|X^{(s)})\le L\sqrt{d}3^{d+1}s^{-1/4}.$$
\end{lemma}
\begin{proof}
In the canonical coupling the event $H_{s,z}$ expresses that
$$X\cap\{(x,u):\, x\in Q_{3s'}(s'z)\text{ and }\lambda^{(s)}(x)\le u\le\lambda_{s,z}\}=\es.$$
Conditioned on $X^{(s)}$ the number of points in the left-hand side is a Poisson random variable whose parameter is bounded above by $(\lambda_{s,z}-\lambda_{s,z,\ms{min}})3^d\sqrt{s}$,
where 
$$\lambda_{s,z,\ms{min}}=\min\big\{\lambda^{(s)}(x):\,x\in Q_{3s'}(s'z)\big\}.$$
Since $\lambda$ is locally Lipschitz, we obtain that almost surely,
$$1-\P(H_{s,z}|X^{(s)})\le L\sqrt{d}3^{d+1}s^{-1/2+1/(2d)}.$$
We conclude the proof by observing that $-1/2+1/(2d)\le -1/4$.
\end{proof}

Now, we can complete the proof of Lemma~\ref{badBoundLem}.

\begin{proof}[Proof of Lemma~\ref{badBoundLem}]
First, we note that conditioned on $X^{(s)}$ the process of $s$-good sites is $3$-dependent with respect to the sub-cubes with side length $s'$. 
More precisely, using a standard construction in dependent percolation, we let $\g=\{z\in\Z^d:\, Q_{s'}(s'z)\cap\G\neq\es\}$ be the discretization of $\G$ and define $M_i=z_i+3\Z^d$, where $z_i\in\Z^d$ can be chosen such that $\{M_1,\dots,M_K\}$ is a partition of $\Z^d$ for $K=3^d$. 

%
%

Hence, for every $i$ conditioned on $X^{(s)}$ the process of $s$-good sites is an independent site process on $\g_i=\g\cap M_i$. 
Moreover, by Lemma~\ref{whpHpLem}, conditioned on $X^{(s)}$ the probability for a site to be $s$-bad is of order $O(s^{-1/4})$. Let $\#_{\ms{b}}\g$ denote the number of bad sites in $\g$ then, having shown Lemma~\ref{whpHpLem}, we may now apply  the Binomial concentration inequality~\cite[Lemma 1.1]{penrose}. This implies that almost surely under the event $\{\#\Gamma\le3^{d+1} s(s')^{-1}\}$ we have that
\begin{align*}
	\P(\#_{s,\ms{b}}\Gamma\ge s^{1-5/(8d)}|X^{(s)})\le \sum_{i=1}^K\P(\#_{\ms{b}}\g_i\ge K^{-1}s^{1-5/(8d)}|X^{(s)})\le\exp(-\sqrt{s}).&\qedhere
\end{align*}
\end{proof}

The proof of Lemma~\ref{linGrowthLem} is based on two main ideas. First, we use a percolation-type argument to show that in most of the cubes that are intersected by a trajectory, the point process $X^{(s)}$ cannot be distinguished from a homogeneous Poisson point process. Second, we use the fluctuation results from~\cite{BB:2007} to show that the total number of such cubes is of order $O(s^{1-1/(2d)})$.

From now on, we consider paths in the graph $\Z^d$ with edges between sites of $d_\infty$-distance $1$.
\begin{lemma}
\label{linPercLem}
Let $E^{(4)}_s$ denote the event that there exists a finite connected set $\gamma$ in $\Z^d$ such that 
\begin{enumerate}
\item $\#\gamma\ge\sqrt{s}$,
\item $s'z\in D_{-,s}$ holds for every $z\in\gamma$,
\item the number of $s$-good sites intersected by $\gamma$ is at most $\#\gamma/2$.  
\end{enumerate}
Then $\lim_{s\to\infty}\P(E^{(4)}_s)=0$.
\end{lemma}
\begin{proof}
Since the process of $s$-good sites is a $3$-dependent percolation process, Lemma~\ref{whpHpLem}  allows us to apply~\cite[Theorem 0.0]{domProd}. Hence, the process of $s$-good sites is dominated from below by a Bernoulli site percolation process with probability $p\in(0,1)$ for open sites. Moreover, $p$ can be chosen arbitrarily close to $1$ if $s$ is sufficiently large, so that the claim reduces to a standard problem in Bernoulli percolation theory. For the convenience of the reader, we provide some details on the solution of this problem. For a fixed connected set $\gamma$ the probability that $\gamma$ contains at least $\#\gamma/2$ bad sites is at most 
$2^{\#\gamma}(1-p)^{\#\gamma/2}$.
Moreover, by~\cite[Lemma 9.3]{penrose}, the number of connected sets containing $k\ge1$ sites is bounded above by $s^d2^{3^dk}$. Therefore, 
$$\P(E^{(4)}_s)\le s^d\sum_{k\ge \sqrt{s}}2^{3^dk} 2^k(1-p)^{k/2},$$
which is of order $O(s^{-2d})$, provided that $p$ is chosen sufficiently close to 1.
\end{proof}
Finally, we need an elementary deterministic result giving an upper bound on the number of $s$-good cubes of a path in terms of its horizontal extent.
\begin{lemma}
\label{deterLem}
Suppose that $(E^{(1)}_s\cup E^{(2)}_s)^c$ occurs and let $X_i\in X^{(s)}$ be arbitrary. Furthermore, let $X_{\ms{end}}\in X^{(s)}$ be the end point of $\Gamma_{-,s}(X_i)$. Then, 
$\pi_1(X_{\ms{end}}-X_i)\ge3^{-d-2}s'\#_{s,\ms{g}}\Gamma_{-,s}(X_i)$.
\end{lemma}
\begin{proof}
Let $\gamma$ be a subset of $s$-good sites whose cubes are intersected by $\Gamma_{-,s}(X_i)$ such that every pair of distinct sites in $\gamma$ is of $d_\infty$-distance at least $3$ and $\#\gamma\ge3^{-d}\#_{s,\ms{g}}\Gamma_{-,s}(X_i)$. Writing $k=\#\gamma$ and $\gamma=\{z_1,\ldots,z_k\}$, we now define subpaths $\Gamma_1,\ldots \Gamma_k$ of $\Gamma$, where the starting point $X_{j,0}$ of $\Gamma_j$ is the first point of $\Gamma$ that is contained in the cube $Q_{s'}(s'z_j)$. Starting from that point, $\Gamma_j$ is the longest subpath of $\Gamma$ that is contained in the left half-space $(-\infty,\pi_1(s'z_j)+s')\times\R^{d-1}$. Since the events $E^{(1)}_s$ and $E^{(2)}_s$ do not occur, we conclude that these subpaths are disjoint and, moreover, that $\pi_1(X_{j,\ms{end}}-X_{j,0})\ge s'/4$, where $X_{j,\ms{end}}$ denotes the endpoint of $\Gamma_j$.
Combining these lower bounds shows that 
\begin{align*}
	\pi_1(X_{\ms{end}}-X_i)\ge\tfrac{s'}{4}\#\gamma\ge3^{-d-2}s'\#_{s,\ms{g}}\Gamma_{-,s}(X_i),\qedhere
\end{align*}
\end{proof}
Using Lemma~\ref{linPercLem} and Lemma~\ref{deterLem}, the proof of Lemma~\ref{linGrowthLem} is now elementary.
\begin{proof}[Proof of Lemma~\ref{linGrowthLem}]
We claim that $E^{(3)}_s\subset E^{(1)}_s\cup E^{(2)}_s\cup E^{(4)}_s$. Indeed, assume that the complement of the event $E^{(1)}_s\cup E^{(2)}_s\cup E^{(4)}_s$ occurs. In order to derive a contradiction, we assume that there exists $X_i\in X^{(s)}$ such that $\#_{s}\Gamma_{-,s}(X_i)\ge3^{d+3}s(s')^{-1}$. Since the complement of the event $E^{(4)}_s$ occurs, we obtain that 
$$\#_{s,\ms{g}}\Gamma_{-,s}(X_i)\ge \tfrac12\#_s\Gamma_{-,s}(X_i)\ge3^{d+2}s(s')^{-1}.$$
In particular, Lemma~\ref{deterLem} would then imply that $\pi_1(X_{\ms{end}}-X_i)\ge s$.
But this is impossible, since both $X_1$ and $X_{\ms{end}}$ are contained in $sD$, a cube of side length $s$.
\end{proof}